\documentclass[12pt]{article}
\usepackage[margin=2.3cm,a4paper]{geometry}
\usepackage{amsmath}
\usepackage{amssymb}
\usepackage{amsthm}
\usepackage[utf8]{inputenc}
\usepackage{fullpage}
\usepackage{tabularx}
\usepackage{multirow}
\usepackage{rotating}
\usepackage{enumerate}

\usepackage{graphicx}  
\usepackage{caption,subcaption}

\usepackage[algo2e,vlined,longend,linesnumbered,ruled]{algorithm2e}
\usepackage[software,hardware]{mymacros}
\newcommand{\smb}{\left[\begin{smallmatrix}}
\newcommand{\sme}{\end{smallmatrix}\right]}

\usepackage{mathdots}
\usepackage[colorlinks=true]{hyperref}

\theoremstyle{definition}
\newtheorem{defin}{Definition}
\newtheorem{theorem}{Theorem}

\newtheorem{remark}{Remark}
\newtheorem{lemma}[theorem]{Lemma}
\newtheorem{corollary}[theorem]{Corollary}

\newcommand{\expm}[1]{\ensuremath{\mathop{\mathrm{e}^{#1}}}}
\newcommand{\jmax}{\ensuremath{j_{\max}}}

\title{Inexact methods for the low rank solution to large scale Lyapunov equations} %
\author{Patrick K\"{u}rschner\thanks{Max Planck Institute for
Dynamics of Complex Technical Systems,  Computational Methods in  Systems and
Control Theory, Magdeburg, Germany, {\tt
  kuerschner@mpi-magdeburg.mpg.de}} \and Melina A. Freitag \thanks{Department of Mathematical Sciences,
		University of Bath, Claverton Down, BA2 7AY, United Kingdom, {\tt
  m.a.freitag@bath.ac.uk}}}

\begin{document}
\maketitle

\begin{abstract}
The rational Krylov subspace method (RKSM) and the low-rank alternating directions implicit (LR-ADI) iteration are established numerical tools for computing
low-rank solution factors of 
large-scale Lyapunov equations. In order to generate the 
basis vectors for the RKSM, or extend the low-rank factors within the LR-ADI method the repeated solution to a shifted 
linear system is necessary. For very large systems this solve is usually implemented using
iterative methods, leading to inexact solves within this inner iteration. We derive theory for a relaxation 
strategy within these inexact solves, both for the RKSM and the LR-ADI method. Practical choices for relaxing the solution tolerance within the inner linear system are then provided. The theory is supported by several numerical examples.
\end{abstract}

\section{Introduction}
We consider the numerical solution of large scale Lyapunov equations of the form  
\begin{align}\label{lyap}
 AX+XA^T=-BB^T,
\end{align}
where $A\in\Rnn,~B\in\Rnr$. Lyapunov equations play a fundamental role in many areas of applications, such as control theory, model reduction and signal
processing, see, e.g.~\cite{Antoulas05,BenS13}. Here we assume that the spectrum $A$, $\Lambda(A)$, lies in the open left half plane, e.g. $\Lambda(A)\in\C_-$
%($\C_+$ also possible).
and that the right hand side is of low rank, e.g. $r\ll n$, which is often the case in practice. A large and growing amount of literature considers the
solution to~\eqref{lyap}, see~\cite{Simoncini16} and references therein for an overview of developments and methods.

The solution matrix $X$ of~\eqref{lyap} is, in general, dense, making it virtually impossible to store it for large dimensions $n$. 

For a low-rank right hand side, however, the solution $X$ often has
a very small numerical rank~\cite{Pen99,Sab07,Gra04,TruV07,BakES15,BecT17} and many algorithms have been developed that approximate $X$ by a low-rank matrix $X\approx ZZ^T$, where $Z\in
\Rns$, $s\ll n$. Several methods belong to such low-rank algorithms, for instance, projection type methods based on Krylov subspaces~\cite{morJaiK94,Sim07,DruS11,
DruKS11, Sim16} and low-rank alternating directions implicit (ADI) methods~\cite{Pen99,LiW02,BenLT09,BenS13,BenKS14,Kue16}. 

This paper considers both the rational Krylov subspace method (from the family of projection type methods) and the low-rank ADI method. 
One of the computationally most expensive parts in both methods is that, in each iteration step, shifted linear systems of the form
\[
(A-\sigma I)y=z,\quad z\in\Rnr,
\]
need to be solved, where the shift $\sigma$ is usually variable and both the shift $\sigma$ and the right hand side $z$ depend on the particular algorithm used.
Normally these linear systems are solved by sparse-direct or iterative methods. When iterative methods, such as preconditioned Krylov subspace methods, are used
to solve the linear systems, then these solves are implemented inexactly and we obtain a so-called inner-outer iterative method. The outer method is (in our case) a rational Krylov subspace method or a low-rank ADI iteration. The inner problem is the iterative solution to the linear systems. The inner solves are often
carried out at least as accurately as the required solution accuracy for the Lyapunov equation (cf., e.g., the numerical experiments in~\cite{DruKS11}), usually in 
terms of the associated Lyapunov residual norms. It turns out that this is not necessary, and the underlying theory is the main contribution of this paper.

Inexact methods have been considered for the solution to linear systems and eigenvalue problems (see~\cite{SimSz03,EstS04,BouF05,Sim05,FrSp09} and references therein).
One focus has been on inexact inverse iteration and similar methods, where, in general, the accuracy of the inexact solve has to be increased to obtain
convergence~\cite{FrSp08}. For subspace expansion methods, such as the Krylov methods considered in~\cite{BouF05,SimSz03, EstS04,Sim05,GaaS17}, it has been observed
numerically and shown theoretically that it is possible to relax the solve tolerance as the outer iteration proceeds.

In this paper we prove that for both the rational Krylov subspace method and the low-rank ADI algorithm we can relax the stopping  tolerance
within the inner solve, a similar behavior as observed for Krylov methods with inexact
matrix-vector products applied to eigenvalue methods~\cite{Sim05,FrSp09}, linear systems~\cite{SimSz03,EstS04,BouF05}, and matrix functions~\cite{GaaS17}.
%This may not be entirely surprising as we are dealing with subspace expansion methods. However, the proofs are non-trivial. 
To this end, we provide practical relaxation strategies for both methods and give numerical examples.

% extension to more general methods? 
% preconditioners?

The paper is organised as follows. In Section \ref{sec:rksm} we review results about rational Krylov subspace methods. Those are used to show important properties about Galerkin projections. A new inexact rational Arnoldi decomposition is derived, extending the theory of \cite{Ruh94c}. We show in Theorem \ref{th:deltaybound}, Corollary \ref{corYentrybound} and Corollary \ref{cor:deltahybound} that the entries of the solution to the projected Lyapunov equation have a decreasing pattern. This crucial novel result is then used to demonstrate that, in the inexact rational Krylov subspace method, the linear system solve can be relaxed, in a way inversely proportional to this pattern. Section \ref{sec:adi} is devoted to low-rank ADI methods. We show that the low-rank factors arising within the inexact ADI iteration satisfy an Arnoldi like decomposition in Theorem \ref{thm:ADI_ratarndec}. This theory is again new and significant for deriving a specially tailored relaxation strategy for inexact low-rank ADI methods in Theorem \ref{thm:ADI_theorelax}. Finally, in Section \ref{sec:num} we test several practical relaxation strategies and provide numerical evidence for our findings. Our examples show that, in particular for very large problems, we can save up to half the number of inner iterations within both methods for solving Lyapunov equations. 

\paragraph{Notation} Throughout the paper, if not stated otherwise, we use $\|\cdot\|$ to denote the Euclidean vector and associated induced matrix norm. $(\cdot)^*$ stands for the transpose and complex conjugate transpose for real and, respectively, complex matrices and vectors. The identity matrix of dimension $k$ is denoted by $I_k$ with the subscript omitted if the dimension is clear from the context. The $j$th column of the identity matrix is $e_j$. Spectra of a matrices $A$ and matrix pairs $(A,M)$ are denoted by $\Lambda(A)$ and $\Lambda(A,M)$, respectively. The smallest (largest) eigenvalues and singular values of a matrix $X$ are denoted by $\lambda_{\min}(X)$ ($\lambda_{\max}(X)$ and $\sigma_{\min}(X)$ ($\sigma_{\max}(X))$.

%%%%%%%%%%%%%%%%%%%%%%%%%%%%%%%%%%%%%%%%%%%%%%%%%%%%%%%%%%%%%%%%
\section{Rational Krylov subspace methods and inexact solves}
\label{sec:rksm}
\subsection{Introduction and rational Arnoldi decompositions}

Projection methods for~\eqref{lyap} follow Galerkin principles similar to, e.g., the Arnoldi method for eigenvalue computations or linear systems (in this case
called full orthogonalization method).
Let $\cQ=\range{Q}\subset \Cn$ be a subspace of $\Cn$ with a matrix $Q\in\C^{n\times k}$, $k\ll n$ whose columns form an orthonormal basis of $\cQ$: $Q^*Q=I_k$. We look for low-rank approximate solutions to~\eqref{lyap} in the form $Q\tX Q^*$ with $\tX=\tX^*\in\Ckk$, i.e. 
\[
X\in\cZ:=\lbrace Q\tX Q^*\in\Cnn,~\tX=\tX^*,~\range{Q}=\cQ\rbrace.	
\]
The Lyapunov residual matrix for approximate solutions $X\in\cZ$ is
\begin{align}\label{lyapres}
\cR:=\cR(Q\tX Q^*)=A(Q\tX Q^*)+(Q\tX Q^*)A^*+BB^*.
\end{align}
Imposing a Galerkin condition $\cR\perp\cZ$~\cite{morJaiK94,Saa90}
%of solutions of such form
%$\cR(Q\tX Q^*)=A(Q\tX Q^*)+(Q\tX Q^*)A^*+BB^*\perp\cZ$
leads to the projected Lyapunov equation 
\begin{align}\label{projlyap0}
T \tX+\tX T^*+Q^*BB^*Q=0,\quad T:= Q^*AQ,
\end{align}
i.e., $\tX$ is the solution of a small, $k$-dimensional Lyapunov equation which can be solved
by algorithms employing dense numerical linear algebra, e.g., the Bartels-Stewart method~\cite{BarS72}. 
Throughout this article we assume that the compressed problems~\eqref{projlyap0} admit a unique solution which is ensured by
$\Lambda(T)\cap-\Lambda(T)=\emptyset$, which holds, for example, when $\Lambda(T)\subset\C_-$.
A sufficient condition for ensuring $\Lambda(T)\subset\C_-$ is $A+A^*<0$. In practice, however, this condition is rather restrictive because, on the one hand, often $\Lambda(T)\subset\C_-$ holds even for indefinite $A+A^*$ and, one the other hand, the discussed projection methods will still work if the restriction $T$ has eigenvalues in $\C_+$ provided that $\Lambda(T)\cap-\Lambda(T)=\emptyset$ holds.

Usually, one produces sequences of
subspaces of increasing dimensions in an iterative manner, e.g.
$\cQ_1\subseteq\cQ_2\subseteq\ldots\subseteq\cQ_j$. 
For practical problems, using the standard (block) Krylov subspace
\begin{align}\label{standardKrylov}
	\cQ_j=\cK_j(A,B)=\range{\left[B,AB,\ldots,A^{j-1}B\right]}
\end{align}
is not sufficient and leads to a slowly converging process and, hence, large low-rank solution factors.
A better performance can in most cases be achieved by using rational Krylov subspaces which we use here in the form
\begin{align}
 \cQ_j=\cR\cK_j(A,B,\boldsymbol{\xi}):=\range{\left[B,(A-\xi_2 I)^{-1}B,\ldots,\prod\limits_{i=2}^j(A-\xi_i I)^{-1}B\right]},
\end{align}
with shifts $\xi_i\in\C_+$ (hence $\xi_i\notin\Lambda(A)$), $i = 2,\ldots,j$.
An orthonormal basis for $\cR\cK_j$ can be computed by the (block) rational Arnoldi process~\cite{Ruh94c} 
leading to $Q_j\in\C^{n\times jr}$. 
The resulting rational Krylov subspace method (RKSM) for computing approximate solutions to~\eqref{lyap} is given in 
Algorithm~\ref{alg:rksm}.
\begin{algorithm2e}[h]
    \SetEndCharOfAlgoLine{} 
  \SetKwInOut{Input}{Input}
  \SetKwInOut{Output}{Output}
  \caption[Rational Krylov subspace method]{Rational Krylov subspace
    method for solving~\eqref{lyap}}
  \label{alg:rksm}
  % \begin{algorithmic}[1]
  \Input{$A,~B$ as in~\eqref{lyap}, shifts $\{\xi_2,\dots,\xi_{j_{\max}}\}\subset\C_+$, 
    tolerance $0<\tau\ll1$.}
  \Output{$Q\in\C^{n\times jr}$, $\tX\in\C^{jr\times jr}$ such that
    $Q\tX Q^*\approx X$ with $jr\ll n$.}
 Compute thin QR decomposition of $B$: $Q_1\beta=B$, $j=1$.\nllabel{rksmO1}\; 
  \While{$\|\cR(Q_{j}Y_jQ_{j}^*)\|>\tau\|BB^*\|$}{%
 Solve $(A-\xi_{j+1}I)w_{j+1}=q_{j}$ for $w_{j+1}$.\nllabel{rksmLGS}\;
 Orthogonally extend basis
    matrix $Q_{j}$:\\
		$h_{1:j,j}=Q_j^*w_{j+1}$, $\hw_{j+1}=w_{j+1}-Q_{j}h_{1:j,j}$,\\
		$q_{j+1}h_{j+1,j}=\hw_{j+1}$,
$Q_{j+1}=[Q_{j},~q_{j+1}]$.\nllabel{rksmO2}\;
Projected equation: \allowbreak $T_j=Q_j^*(AQ_j)$, $\tB_j=Q_j^*B=[\beta^*,0]^*$.\nllabel{rksmAj}\;%=(I+H_jD_j)H_j^{-1}-Q_j^*Aq_{j+1}e_j^*h_{j+1,j}H_j^{-1}$,\\
Solve    
$T_jY_j+Y_jT_j^*+\tB_j\tB_j^*=0$ for $Y_j$.\nllabel{solveY}\;
Estimate residual norm $\cR(Q_jY_j Q_j^*)$.\nllabel{rksmres}\;
$j=j+1$.\;
}
$\tX = Y_j$, $Q = Q_j$\; 
\end{algorithm2e}
The shift parameters are crucial for a rapid reduction of the Lyapunov residual norm $\|\cR_j\|:=\|\cR(Q_{j}Y_{j}Q_{j}^*)\|$ (cf.~\eqref{lyapres}) and can be generated a-priori or
adaptively in the course of the iteration~\cite{DruS11}. 
The extension of the orthonormal basis $\range{Q_j}$ by $w$ in Line~\ref{rksmO2} of Algorithm~\ref{alg:rksm} should be done by a robust orthogonalization process, e.g., using a repeated
(block) Gram-Schmidt process. 
The orthogonalization coefficients are accumulated in the block Hessenberg matrix $H_j=[h_{i,k}]\in\C^{jr\times jr}$, $h_{i,k}\in\C^{r\times r}$,
$i,k=1,\ldots,j$, $i<k+2$. If the new basis blocks are normalized using a thin QR factorization, then the $h_{i+1,i}$ matrices are upper triangular.

In the following we summarise properties of the RKSM method, in particular known results about the rational Arnoldi decomposition and the Lyapunov residual,
which will be crucial later on. We simplify this discussion of the rational Arnoldi process and Algorithm~\ref{alg:rksm} to the case $r=1$. From there, one can
generalize to the situation $r>1$.  

Assuming at first that the linear systems in Line~\ref{rksmLGS} of Algorithm~\ref{alg:rksm} are solved \emph{exactly}, then, after $j$ steps of
Algorithm~\ref{alg:rksm}, the generated quantities satisfy a rational Arnoldi decomposition~\cite{Ruh94c,Gue13,DruS11}
\begin{align}\label{exactRatArnDecomp}
AQ_j&=Q_jT_j+g_jh_{j+1,j}e_{j}^*H_j^{-1},\quad g_j:=q_{j+1}\xi_{j+1}-(I-Q_jQ_j^*)A q_{j+1},
\end{align}
from which it follows that the restriction of $A$ onto $\cR\cK_j$ can be expressed as
\begin{align}\label{exactProj}
	T_j&=Q_j^*(AQ_j)=(I+H_jD_j)H_j^{-1}-Q_j^*Aq_{j+1}h_{j+1,j}e_j^*H_j^{-1},
\end{align}
with $D_j=\diag{\xi_2,\ldots,\xi_{j+1}}$. $T_j$, $H_j$ and $D_j$ are square matrices of size $j$.
The relation~\eqref{exactProj} can be used to compute the
restriction $T_j$ efficiently in terms of memory usage.  
The rational Arnoldi decomposition can also be expressed as
\begin{align}\label{eq:hesshess}
AQ_{j+1}H_{j+1,j}=Q_{j+1}M_{j+1,j},\quad M_{j+1,j}:=\smb I+H_jD_j\\\xi_{j+1}h_{j+1,j}e_{j}^*\sme,
\end{align}
see, e.g.,~\cite{Ruh94c,BerG15} and $T_j$ can be associated with the Hessenberg-Hessenberg pair $(M_{j,j},H_{j,j})$. Here we generally use $H_j =H_{j,j}$ and
$M_j = M_{j,j}$ for simplicity.  
$H_{j+1,j}$ and $M_{j+1,j}$ are $(j+1)\times j$ matrices.

The final part of this section discusses how the Lyapunov residual in Line~\ref{rksmres} of Algorithm~\ref{alg:rksm} can be computed efficiently. If the
projected Lyapunov equation in Line~\ref{solveY}  
\begin{align}\label{projlyap}
T_jY_j+Y_jT_j^*+\tB_j\tB_j^*&=0
\end{align}
is solved for $Y_j$, then for the Lyapunov residual~\eqref{lyapres} after $j$ steps of the RKSM, we have
\begin{subequations}\label{exactRKres}
\begin{align}\label{exactRKres_R}
\cR_j= \cR(Q_jY_j Q_j^*)&=F_j+F_j^*,
\end{align}
with
\begin{align}\label{exactRKres_F}
	F_j&:=L_jQ_j^*\in\C^{n\times n},\quad L_j:=g_jh_{j+1,j}e_{j}^*H_j^{-1}Y_j\in\C^{n\times j},
\end{align}
\end{subequations}
as was shown in~\cite{DruS11,Sim16}. The term $L_j$ is sometimes referred to as ``semi-residual''. 
Since $g_j\in\Cn$, $h_{j+1,j}e_{j}^*H_j^{-1}Y_j\in\C^{1\times n}$, the matrix $F_j$ is of rank one. In general for $B\in\Rnr$, $r>1$ and a block-RKSM,  $F_j$ is of rank $r$.

Note that relation~\eqref{exactRKres_R} is common for projection methods for~\eqref{lyap}, but the special structure of $L_j$ in 
\eqref{exactRKres_F} arises from the rational Arnoldi process.
Because $g_j^*Q_j=0$ we have that $F_j^2=0$ and $\|\cR_j\|=\|F_j\|=\|L_j\|$ (using the relationship between the spectral norm and the largest eigenvalues/singular values). This enables an efficient way for computing the norm of the Lyapunov residual $\cR_j$ via 
\begin{align}\label{eq:compres}
\|\cR_j\|&=\|\|g_j\|h_{j+1,j}e_{j}^*H_j^{-1}Y_j\|, 
\end{align}
for $g_j\in\Cn$.
%\begin{remark}
%It is possible to avoid the use of the inverse of $H_j$. With $T_j=K_jH_j^{-1}$, where $K_j:=T_jH_j=(I+H_jD_j)-Q_j^*Aq_{j+1}e_j^*h_{j+1,j}$
%from~\eqref{exactProj}, and $Y_j=H_j\hY_jH_j^*$, we have that $\hY_j$ solves the
%generalized Lyapunov equation
%\begin{align}\label{projGlyap}
%	K_j\hY_jH_j^*+H_j\hY_jK_j^*+\tB_j\tB_j^*=0.
%\end{align}
%Consequently, $H_j^{-1}Y_j=\hY_jH_j^*$ and the norm of the Lyapunov residual $\cR_j$ for the projected solution $Y_j$ can be expressed alternatively in terms
%of $\hY_j$ by $\|\cR_j\|=\|\hL_j\|$ with $\hL_j:=g_jh_{j+1,j}e_{j}^*\hY_jH_j^*$.
%\end{remark}

%%%%%%%%%%%%%%%%%%%%%%%%%%%%%%%%%%%%%%%%%%%%%%
\subsection{The inexact rational Arnoldi method}
\label{sec:RKSMinexact}
The solution of the linear systems for $w_{j+1}$ in each step of RKSM (line~\ref{rksmLGS} in Algorithm~\ref{alg:rksm}) is one of the  computationally most expensive
stages  of the rational Arnoldi method. In this section we investigate inexact solves of this linear systems, e.g., by iterative Krylov subspace methods, but we
assume that this is the only source of inaccuracy in the algorithm. 
% This leads to a so-called \emph{residual gap} between
% the true and computed residual that can be used for developing rules for the required accuracy of these inexact linear solves.
Clearly, some of the above properties do not hold anymore if the linear systems are solved inaccurately. Let
\begin{align*}
s_j:=q_j-(A-\xi_{j+1} I)\tw_{j+1},\quad\text{with}\quad\|s_j\|\leq \tau^{\text{LS}}_j 
\end{align*}
be the residual vectors with respect to the linear systems and inexact solution vectors $\tw_{j+1}$ with $\tau^{\text{ls}}_j<1$ being the solve tolerance of
the linear system at step $j$ of the rational Arnoldi method.

The derivations in~\cite{Ruh94c},\cite[Proof of Prop. 4.2.]{DruS11} can be modified with 
\begin{align*}
q_j&=(A-\xi_{j+1}I)Q_{j+1}h_{1:j+1,j}+s_j,
\end{align*}
in order to obtain 
\begin{align*}
AQ_{j+1}H_{j+1,j}=Q_{j+1}M_{j+1,j}-S_j,~S_j:=[s_1,\ldots,s_j],
\end{align*}
where $M_{j+1,j}$ is as in~\eqref{eq:hesshess}. Note that for $S_j=0$ we recover~\eqref{eq:hesshess}, as expected. Hence $AQ_jH_j=[Q_j(I+H_jD_j)+(\xi_{j+1}q_{j+1}-Aq_{j+1})e^*_jh_{j+1,j}]-S_j$ which leads to the perturbed rational Arnoldi relation
\begin{align}\label{inexactRatArnDecomp1}
	AQ_j&=Q_jT^{\text{impl.}}_j+g_je_j^*h_{j+1,j}H_j^{-1}-S_jH_j^{-1},
\end{align}
where $T^{\text{impl.}}_j:=[(I+H_jD_j)-Q_j^*Aq_{j+1})e^*_jh_{j+1,j}]H_j^{-1}$ marks the restriction of $A$ by the implicit formula~\eqref{exactProj}.
However, the explicit restriction of $A$ can be written as
\begin{align}\label{inexactProj}
\begin{split}
T^{\text{expl.}}_j&:=Q_j^*(AQ_j)=T^{\text{impl.}}_j-Q_j^*S_jH_j^{-1},\\
&=((I+H_jD_j)-Q_j^*Aq_{j+1})e^*_jh_{j+1,j}-Q_j^*S_j)H_j^{-1}
\end{split}
\end{align}
which highlights the problem that the implicit (computed) restriction $T^{\text{impl.}}_j$ from~\eqref{exactProj} is not the true restriction of $A$
onto $\range{Q_j}$.
In fact, the above derivations reveal that
\begin{align*}
	T^{\text{impl.}}_j=Q_j^*(A+E_j)Q_j,\quad E_j:=S_jH_j^{-1}Q_j^*,
\end{align*} 
i.e., the implicit restriction~\eqref{exactProj} is the exact restriction of a perturbation of $A$~\cite{Gue10,LehM98}.

Similar to~\cite{Gue10} we prefer to use $T^{\text{expl.}}_j$ for defining the projected problem as this keeps the whole process slightly closer to the original
matrix $A$. 
% In the language of~\cite{Gue10} the obtained approximations $Q_jY_jQ_j^*$ are called \textit{corrected approximations}. 
Moreover,~\eqref{inexactProj} reveals that unlike $T^{\text{impl.}}$, the explicit restriction
$T^{\text{expl.}}_j$ is not connected to a Hessenberg-Hessenberg pair because the term $Q_j^*S_j$ has no Hessenberg structure.
Computing $T^{\text{expl.}}_j$ by either~\eqref{inexactProj} or explicitly generating $T^{\text{expl.}}_j=Q_j^*(AQ_j)$ by adding new
columns and rows to $T^{\text{expl.}}_{j-1}$ will double the memory requirements because an additional $n\times j$ array, $S_j\in\C^{n\times j}$ or $W_j:=AQ_j$, has to be stored. Hence, the per-step storage requirements are similar to the extended Krylov subspace method (see, e.g., \cite{Sim07} and the corresponding paragraph in Section~\ref{ssec:imple_gen}).
 Alternatively, if matrix vector products with $A^*$ are available, we can also update the restriction
	\begin{align*}
	T_j^{\text{expl.}}=Q_j^*(AQ_j)=\smb 
	T^{\text{expl.}}_{j-1}&Q_{j-1}^*(Aq_j)\\
	(z_j)^*Q_{j-1}&q_{j}^*(Aq_j)
	\sme,\quad z_j:=A^*q_j,
	\end{align*}
	thereby bypassing the increased storage requirements. This will be the method of choice in our numerical experiments.
	Trivially, we may also compute $T_j^{\text{expl.}}$ from scratch row/columnwise, requiring $j$ matrix vector products with $A$ but circumventing the storage increase.

Using $T^{\text{expl.}}_j$ leads to the \emph{inexact rational Arnoldi relation}
\begin{align}\label{inexactRatArnDecomp2}
AQ_j&=Q_jT^{\text{expl.}}_j+\tg_jH_j^{-1},\quad \tg_j:=g_jh_{j+1,j}e_j^*-(I-Q_jQ_j^*)S_j,
\end{align}
which we employ in the subsequent investigations.
\begin{lemma}
Consider the approximate solution $Q_jY_jQ_j^*$ of~\eqref{lyap} after $j$ iterations of inexact RKSM, where $Y_j$ solves the projected Lyapunov
equation~\eqref{projlyap}
defined by either $T^{\text{expl.}}_j$ or $T^{\text{impl.}}_j$. 
Then the true Lyapunov residual matrix $\cR^{\text{true}}_j=\cR(Q_jY_jQ_j^*)$ can be written in the following forms.
\begin{subequations}\label{inexactRKres_R}
\begin{enumerate}
 \item[(a)] If $T^{\text{expl.}}_j$ is used it holds
 \begin{align}\label{inexactRKres_FE}
\cR^{\text{true}}_j &=F^{\text{expl.}}_j+(F^{\text{expl.}}_j)^*,\quad F^{\text{expl.}}_j:=\tg_jH_j^{-1}Y_jQ_j^*=F_j-(I-Q_jQ_j^*)S_jH_j^{-1}Y_jQ_j^*, 
\end{align}
\item[(b)] and, if otherwise $T^{\text{impl.}}_j$ is used, it holds
 \begin{align}\label{inexactRKres_FI}
\cR^{\text{true}}_j &=F^{\text{impl.}}_j+(F^{\text{impl.}}_j)^*,\quad F^{\text{impl.}}_j:=F_j-S_jH_j^{-1}Y_jQ_j^*. 
\end{align}
\end{enumerate}
\end{subequations}
\end{lemma}
\begin{proof}
For case (a), using~\eqref{inexactRatArnDecomp2} immediately yields
 \begin{align*}
\cR^{\text{true}}_j &= \cR(Q_jY_j Q_j^*)=AQ_jY_jQ_j^*+Q_jY_jQ_j^*A^*+BB^*\\
&=[Q_jT^{\text{expl.}}_j+\tg_jH_j^{-1}]Y_jQ_j^*+Q_jY_j[Q_jT^{\text{expl.}}_j+\tg_jH_j^{-1}]^*+Q_jQ_j^*BB^*Q_jQ_j^*\\
&=Q_j\left[T^{\text{expl.}}_jY_j+Y_j(T^{\text{expl.}}_j)^*+Q_j^*BB^*Q_j\right]Q_j^*+\tg_jH_j^{-1}Y_jQ_j^*+Q_jY_jH_j^{-*}\tg^*_j=F^{\text{expl.}}_j+(F^{\text{expl.}}_j)^*.
\end{align*}
Case (b) follows similarly using~\eqref{inexactRatArnDecomp1}.
\end{proof}
Hence in both cases the true residual $\cR^{\text{true}}_j$ is a perturbation of the computed residual given
by~\eqref{exactRKres_R}-\eqref{exactRKres_F} which we denote in the remainder by $\cR^{\text{comp.}}_j$. 
In case $T^{\text{expl.}}_j$ is used, since $Q_j\perp \tg_j$ one can easily see that $\|\cR^{\text{true}}_j\|=\|F^{\text{expl.}}_j\|=\|\tg_jH_j^{-1}Y_j\|$, a property not
shared when using $T^{\text{impl.}}_j$. 
% Of course, using this formula would again require storing $S_j$ which makes working with the true residual impractical during the computations. 

% Here we use
% $T^{\text{expl.}}_j$ since the bounds for $Y_j$ and $H_j\hY_j$ involving $\|\cR^{\text{true}}_{k-1}\|$ in Section~\ref{sec:props} require the use of this
% restriction.
Our next step is to analyze the difference between the true and computed residual. 
\begin{defin}
The \textit{residual gap} after $j$ steps of inexact RKSM for Lyapunov equations is given by
\begin{align*}
	\Delta\cR_j&:=\cR_j^{\text{true}}-\cR^{\text{comp.}}_j = \eta^{\text{expl./impl.}}_j+(\eta^{\text{expl./impl.}}_j)^*,\\
	\text{where}\quad \eta^{\text{expl.}}_j&:=F_j-F^{\text{expl.}}_j=(I-Q_jQ_j^*)S_jH_j^{-1}Y_jQ_j^*\quad(\text{if}~T^{\text{expl.}}_j~\text{is used}),\\
	\eta^{\text{impl.}}_j&:=F_j-F^{\text{impl.}}_j=S_jH_j^{-1}Y_jQ_j^*\quad(\text{if}~T^{\text{impl.}}~\text{is used}).
\end{align*}
%where $\eta_j:=F_j-\tF_j=(I-Q_jQ_j^*)S_jH_j^{-1}Y_jQ_j^*$.
\end{defin}
We have
\begin{align*}
\|\Delta\cR_j\|=\begin{cases}
\|\eta^{\text{expl.}}_j+(\eta^{\text{expl.}}_j)^*\|&=\|\eta^{\text{expl.}}_j\|,\\
\|\eta^{\text{impl.}}_j+(\eta^{\text{impl.}}_j)^*\|&\leq 2\|\eta^{\text{impl.}}_j\|.
\end{cases}
\end{align*}
The result for $\|\eta^{\text{expl.}}_j\|$ follows from the orthogonality of left and rightmost factors of $\eta^{\text{expl.}}_j$:
with $\hQ:=[Q_j,Q_j^{\perp}]\in\Cnn$ unitary we have
\begin{align*}
\|\Delta\cR_j\|_2&=\vert\lambda_{\max}(\eta^{\text{expl.}}_j+(\eta^{\text{expl.}}_j)^*)\vert=\vert\lambda_{\max}\left(\hQ^*(\eta^{\text{expl.}}_j+(\eta^{\text{expl.}}_j)^*)\hQ\right)\vert\\
&=\vert\lambda_{\max}(\smb 0&{Q_j^{\perp}}^*\eta^{\text{expl.}}_jQ_j\\({Q_j^{\perp}}^*\eta^{\text{expl.}}_jQ_j)^*&0\sme)\vert=\sigma_{\max}(\eta^{\text{expl.}}_j)=\|\eta^{\text{expl.}}_j\|_2.
\end{align*}
In the following we use $T^{\text{expl.}}_j$ to define the projected problem and omit the superscripts $^{\text{expl.}}$, $^{\text{impl.}}$.

Suppose the desired accuracy is so that $\|\cR^{\text{true}}_j\|\leq\varepsilon$, where $\varepsilon>0$ is a given threshold. 
In practice the computed residual
norms often show a decreasing behavior very similar to the exact method. However, the norm of the residual gap $\|\eta_j\|$ indicates the attainable accuracy of
the inexact rational Arnoldi method because
$\|\cR^{\text{true}}_j\|\leq\|\cR^{\text{comp.}}_j\|+\|\eta_j\|$ and the true residual norm is bounded by $\|\eta_j\|$ even if
$\|\cR^{\text{comp.}}_j\|\leq\varepsilon$, which would indicate convergence of the computed residuals. This motivates to enforce $\|\eta_j\|<\varepsilon$, such
that small
true residual norms $\|\cR^{\text{true}}_j\|\leq 2\varepsilon$ are obtained overall. Since, at step $j$,
\begin{align}\label{rksmgap:bound1}
\|\eta_j\|&\leq \|S_jH_j^{-1}Y_j\|=\|\sum\limits_{k=1}^js_ke_k^*H_j^{-1}Y_j\|\leq\sum\limits_{k=1}^j\|s_k\|\|e_k^*H_j^{-1}Y_j\|,
\end{align}
it is sufficient that only one of the factors in each addend in the sum is small and the other one is bounded by, say,
 $\cO(1)$ in order to achieve $\|\eta_j\|\leq\varepsilon$. In particular, if the $\|e_k^*H_j^{-1}Y_j\|$ terms decrease with $k$, the linear residual norms
$\|s_k\|$ are
allowed to increase to some extent, and still achieve a small residual gap $\|\eta_j\|$. Hence, the solve tolerance $\tau^{\text{ls}}_k$
 of the linear solves can be relaxed in the course of the outer iteration which has coined the term \textit{relaxation}. For this to happen, however, we first
need to investigate if there is a decay of $\|e_k^*H_j^{-1}Y_j\|$ as $k$ increases.
%%%%%%%%%%%%%%%%%%%%%%%%%%%%%%%%%%%%%%%%%%%%%%%%%%%%%%%%%%%%%%%%%%%%%%%%%%%%%%%%%%%%%%%%%%%%%%
\subsection{Properties of the solution of the Galerkin system}
\label{sec:props}
% Before we consider inexact solves in Line~\ref{rksmLGS} of Algorithm~\ref{alg:rksm} 
By~\eqref{rksmgap:bound1}, the norm of the residual gap $\eta_j$ strongly depends on  
% we investigate the properties of 
the solution  $Y_j$ of the projected Lyapunov equation~\eqref{projlyap}. We will see in Theorem~\ref{th:deltaybound} and Corollary~\ref{corYentrybound} that the
entries of $Y_j$ decrease
away from the diagonal in a manner proportional to the Lyapunov residual norm. 

In the second part of this section we consider the rows of $H_j^{-1}Y_j$, since the residual formula
(\ref{exactRKres_R}-\ref{exactRKres_F}) and the residual gap~\eqref{rksmgap:bound1} depend on
this quantity. It turns out that the norm of those rows also decay with the Lyapunov residual norms. Both decay bounds will later be used to develop practical
relaxation criteria for achieving $\|\eta_j\|\leq\varepsilon$.

Consider the solution to the projected Lyapunov equation~\eqref{projlyap}. We are interested in the transition from step $k$ to $j$, where $k<j$. At first,
we investigate this transition for a general Galerkin method including RKSM as a special case.
\begin{theorem}\label{th:deltaybound}
Assume $A+A^*<0$ and consider a Galerkin projection method for~\eqref{lyap} 
with $T_j=Q_j^*(AQ_j)$, $Q_j^*Q_j=I_j$, and the first basis vector given by $B=q_1\beta$.
Let the $k\times k$ matrix $Y_k$ and the $j\times j$ matrix $Y_j$ be the solution to the projected Lyapunov equation~\eqref{projlyap} after $k$ and $j$ steps of
this method (e.g., Algorithm~\ref{alg:rksm} with $r=1$), respectively, where $k<j$. Consider the $j\times j$ difference matrix $\Delta Y_{j,k}:=Y_{j}-\smb
Y_k&0\\0&0\sme$, where the
zero blocks are of appropriate size. Then 
\begin{align}\label{DYbound}
\|\Delta Y_{j,k}\|&\leq c_A\|\cR^{\text{true}}_k\|,\quad c_A:=\frac{(1+\sqrt{2})^2}{2\alpha_A},
\end{align}
where $\alpha_A:=\half\vert\lambda_{\min}(A+A^*)\vert$, and $\cR^{\text{true}}_k$ is the Lyapunov residual matrix after $k$ steps of Algorithm~\ref{alg:rksm}.
\end{theorem}
\begin{proof}
The residual matrix $N_{j,k}$ of~\eqref{projlyap} w.r.t. $T_j$ and $\smb Y_k&\\&0\sme$ is given by
\begin{align*}
N_{j,k}:=T_{j}\smb Y_k&\\&0\sme+\smb Y_k&\\&0\sme T_{j}^*+\smb\smb\beta\beta^*&0\\0&0\sme&0\\0&0\sme.%=\smb 0&Y_kt_2^*\\t_2Y_k&0\sme,
\end{align*}
%where $t_2 = [q_{k+1},\ldots,q_j]^*AQ_k$. 
as $T_j$ is built cumulatively. Since $Q_j\smb Y_k&\\&0\sme Q_j^*=Q_kY_kQ_k^*$, it follows that $\|N_{j,k}\|\leq \|\cR^{\text{true}}_k\|_2$.
The difference matrix $\Delta Y_{j,k}$ satisfies the Lyapunov equation
$T_{j}\Delta Y_{j,k}+\Delta Y_{j,k}T_{j}^*=-N_{j,k}$, and since $\Lambda(T_j)\subset\C_-$ it can be expressed via the integral
\begin{align}
\label{eq:dybound}
	\Delta Y_{j,k}&=\intab{0}{\infty}\expm{T_{j}t}N_{j,k}\expm{T^*_{j}t}dt.
\end{align}
Moreover, using results from~\cite{CroP17}, we have 
\[
\|\expm{T_{j}t}\|\leq (1+\sqrt{2})\max\limits_{z\in\cW(T_j)}\vert \expm{zt}\vert=(1+\sqrt{2})\expm{\max\limits_{z\in\cW(T_j)}\Real{z} t},
\]
where $\cW(\cdot)$ denotes the field of values.
Using the assumption $A+A^*<0$ it holds that $\cW(A)\in\C_-$ and, consequently, $\cW(T_{j})\subseteq\cW(A),~\forall j>0$. Hence,
\[
\expm{\max\limits_{z\in\cW(T_j)}\Real{z} t} \leq \expm{-\alpha_A t}\quad\text{with}\quad\alpha_A:=\half\vert\lambda_{\min}(A+A^*)\vert.
\]
Finally, from~\eqref{eq:dybound} we obtain $\|\Delta Y_{j,k}\|\leq \displaystyle\intab{0}{\infty}\|\expm{T_{j}t}\|^2dt \|N_{j,k}\| \leq
\displaystyle\frac{(1+\sqrt{2})^2\|\cR^{\text{true}}_k\|}{2\alpha_A}$.
\end{proof}
Theorem \ref{th:deltaybound} shows that the difference matrix $\Delta Y_{j,k}$ decays at a similar rate as the Lyapunov residual norms.
\begin{remark} Note that the constant $c_A$ in Theorem~\ref{th:deltaybound} is a bound on $\intab{0}{\infty}\|\expm{T_{j}t}\|^2dt$ and could be much smaller than the given value. In the following considerations we assume that $c_A$ is small (or of moderate size).
If $c_A$ is large (e.g. the field of values close to the imaginary axis) then the ability to approximate $X$ by a low-rank factorization might be declined, see, e.g.~\cite{Gra04,BakES15}. Hence, this situation could be difficult for the low-rank solvers (regardless of exact or inexact linear solves), and it is therefore  appropriate to assume that $c_A$ is small.
\end{remark}
The above theorem can be used to obtain results about the entries of the solution $Y_j$ of the projected Lyapunov equation~\eqref{projlyap}.
\begin{corollary}\label{corYentrybound} Let the assumptions of Theorem~\ref{th:deltaybound} hold. For the $(\ell,i)$-th entry of $Y_j$ we have 
\begin{align}\label{colYbound}
	%\vert e_{\ell}^*Y_je_i\vert \leq c_A\min\limits_{p=1,\ldots,\max(\ell,i)}\|\cR^{\text{true}}_{p-1}\|,\quad \ell,i=1,\ldots,j,
\vert e_{\ell}^*Y_je_i\vert \leq c_A\|\cR^{\text{true}}_{k}\|,\quad \ell,i=1,\ldots,j,
\end{align}
where $k< \max(\ell,i)$ and  $c_A$ is as in Theorem~\ref{th:deltaybound}.
\end{corollary}
\begin{proof}
We have 
\begin{align}
	e_{\ell}^*Y_je_i=e_{\ell}^*\smb Y_k&0\\0&0\sme e_i+e_{\ell}^*\Delta Y_{j,k}e_i.
\end{align}
For $k<\ell$ the first summand vanishes and hence, using Theorem~\ref{th:deltaybound}
\begin{align}
	\vert e_{\ell}^*Y_je_i\vert=\vert e_{\ell}^*\Delta Y_{j,k}e_i\vert \leq \frac{(1+\sqrt{2})^2}{2\alpha_A}\|\cR^{\text{true}}_k\|,\quad k=1,\ldots,\ell-1.
\end{align}
Since, $Y_j=Y_j^{*}$ the indices $\ell$, $i$ can be interchanged, s.t. $k< \max(\ell,i)$ and we end up with the final bound \eqref{colYbound}.
\end{proof}
\begin{remark}
Note that the sequence $\lbrace\|\cR^{\text{true}}_k\|\rbrace$ is not necessarily monotonically decreasing and we can further extend the bound in
\eqref{colYbound} and obtain
\[
\vert e_{\ell}^*Y_je_i\vert \leq c_A\min\limits_{p=1,\ldots,\max(\ell,i)}\|\cR^{\text{true}}_{p-1}\|,\quad \ell,i=1,\ldots,j.
\]
\end{remark}
Corollary~\ref{corYentrybound} shows that the $(\ell,i)$-th entry of $Y_j$ can be bounded from above using the Lyapunov residual norm at step $k<\max(\ell,i)$.
This indicates that the further away from the diagonal, the smaller the entries of $Y_j$ will become in the course of the iteration, provided the true
residual norms exhibit a sufficiently decreasing behavior. We can observe this characteristic in Figure~\ref{fig:decayY} in Section~\ref{sec:num}. The decay of
Lyapunov solutions has also been investigated by different approaches. Especially when
$T_j$ is banded, which is, e.g., the case when a Lanczos process is applied to $A=A^*$, more refined decay bounds for the entries of $Y_j$ can be established, see, e.g.,~\cite{CanSV14,PalS17}.

In exact and inexact RKSM, the formula for residuals~\eqref{exactRKres},~\eqref{inexactRKres_R} and residual gaps~\eqref{rksmgap:bound1}
depend not only on $Y_j$ but rather
on $H_j^{-1}Y_j$. In particular, the rows of $H_j^{-1}Y_j$ appear in~\eqref{rksmgap:bound1} and will later be
substantial for defining relaxation strategies. Hence, we shall investigate if the norms $\|e_k^*H_j^{-1}Y_j\|$ also exhibit a decay for $1\leq k\leq j$.
For the last row, i.e. $k=j$, using~\eqref{eq:compres} readily reveals
\begin{align*}
 \|e_j^*H_j^{-1}Y_j\|\leq \frac{\|\cR^{\text{comp.}}_j\|}{\vert h_{j+1,j}\vert\|g_j\|},\quad \text{assuming}\quad g_j\neq 0,\,h_{j+1,j}\neq 0.
\end{align*}
In the spirit of Theorem~\ref{th:deltaybound}, Corollary~\ref{corYentrybound}, we would like to bound the $\ell$th row of $H_j^{-1}Y_j$ by the $(\ell-1)$th
computed Lyapunov residual norm~$\|\cR^{\text{comp.}}_{\ell-1}\|$. For this we require the following lemma showing that, motivated by similar results
in~\cite{EstS04}, the
first $k$ ($1< k\leq j$) entries of $e_k^*H_j^{-1}$ are essentially determined by the left null space of
$\underline{H}_j:=H_{j+1,j}$ and $e_{k-1}^*H_{k-1}^{-1}$ modulo scaling.
\begin{lemma}\label{Lem:ekinvHj}
 Let  $\underline{H}_j=\smb H_j\\h_{j+1,j}e_j^*\sme\in\C^{(j+1)\times j}$ with $H_j$ an unreduced upper Hessenberg matrix, with rank$(\underline{H}_j)=j$, $H_k:=H_{1:k,1:k}$ nonsingular $\forall 1\leq k\leq j$, and 
 let $\omega\in\C^{1\times (j+1)}$ satisfy $\omega \underline{H}_j=0$. Define the vectors
$f^{(k)}_m:=e_k^*H_m^{-1}$, $1\leq k,m\leq j$. Then, for $1\leq k\leq j$, we have 
\begin{subequations}
 \begin{align}\label{ekinHj_fjj}
 f^{(j)}_j&=-\frac{\omega_{1:j}}{\omega_{j+1}h_{j+1,j}},\quad
f^{(k)}_j=\frac{v^{(k)}_j}{\phi_{j}^{(k)}},\quad \text{where} \\\label{ekinHj_fjk}
% \begin{split}
%  \quad \text{where} \\
  v^{(k)}_j&:=\omega_{1:j}+[0_{1,k},[0_{1,j-k-1},h_{j+1,j}\omega_{j+1}]H_{k+1:j,k+1:j}^{-1}],\quad\phi_{j}^{(k)}:=v^{(k)}_jH_{j}e_k.
% \end{split}
\end{align}
Moreover, for $k>1$, the first $k$ entries of $v^{(k)}_j$ can be expressed by
 \begin{align}\label{ekinHj_fold}
  (v^{(k)}_j)_{1:k}&=[-h_{k,k-1}f^{(k-1)}_{k-1},1]\omega_k.
\end{align}
\end{subequations}
\end{lemma}
\begin{proof}
See Appendix~\ref{sec:appA}.
\end{proof}

% A generalization of Corollary~\ref{corYentrybound} is
%\textbf{Nur grober Ansatz, muss noch vervollständigt werden.}
\begin{corollary}
\label{cor:deltahybound}
Let the assumptions of Theorem~\ref{th:deltaybound}, Corollary~\ref{corYentrybound}, and Lemma~\ref{Lem:ekinvHj} hold and consider the RKSM as a special Galerkin
projection method. Let $\omega\in\C^{1\times (j+1)}$ be a unit
vector with $\omega \underline{H}_{j}=0$, where $\underline{H}_{j}=\smb H_j\\e_{j}^*h_{j+1,j}\sme$ is the upper Hessenberg matrix obtained after $j$ steps
of the rational
Arnoldi decomposition. Then
 \begin{align} 
\|e_{\ell}^*H_j^{-1}Y_j\|\leq\begin{cases}
c_A\|e_{1}^*H_j^{-1}\|\|\cR^{\text{true}}_{0}\|,&\ell=1,\\
\frac{1}{\phi_{j}^{(\ell)}\|g_{\ell-1}\|}\|\cR^{\text{comp.}}_{\ell-1}\|+c_A\|e_{\ell}^*H_j^{-1}\|\|\cR^{\text{true}}_{\ell-1}\|,& \ell=2,\ldots,j
\end{cases}
 \end{align}
with $g_{\ell}$ from~\eqref{exactRatArnDecomp} and $\phi_{j}^{(\ell)}$ from Lemma~\ref{Lem:ekinvHj}. 
\end{corollary}
\begin{proof}
 Using Lemma~\ref{Lem:ekinvHj} for $f_j^{(\ell)}:=e_{\ell}^*H_j^{-1}$, $1<\ell\leq j$, and the structure~\eqref{exactRKres_F} of the computed Lyapunov residuals
yields
 \begin{align*}
  e_{\ell}^*H_j^{-1}Y_j&=f_j^{(\ell)}\left(\smb Y_{\ell-1}&0\\0&0\sme+\Delta Y_{j,\ell-1}\right)
=[-h_{\ell,\ell-1}f_{\ell-1}^{(\ell-1)}\omega_{\ell}]/\phi_{j}^{(\ell)}Y_{\ell-1}+f_j^{(\ell)}\Delta Y_{j,\ell-1}\\
%   &=-h_{\ell,\ell-1}f_{\ell-1}^{(\ell-1)}Y_{\ell-1}/\phi_{j}^{(\ell)}+f_j^{(\ell)}\Delta Y_{j,\ell-1}\\
   &=-\frac{g_{\ell-1}^*L_{\ell-1}}{\|g_{\ell-1}\|^2\phi_{j}^{(\ell)}}\omega_{\ell}+f_j^{(\ell)}\Delta Y_{j,\ell-1}.
 \end{align*}
Taking norms, using that $\omega$ is a unit norm vector, noticing for $\ell=1$ only the second term exists, and applying Theorem~\ref{th:deltaybound} gives the
result.
\end{proof}
% We have shown that the entries of the solution of the projected Lyapunov equation $Y_j$ decay away from the diagonal, proportionally to the norm of the 
% Lyapunov residuals. 
Corollary \ref{cor:deltahybound} shows that, similar to the entries of $Y_j$, the rows of $H_j^{-1}Y_j$ can be bounded using the previous Lyapunov residual norm. However, due to the influence of
$H_j^{-1}$, the occurring constants 
 in front of the Lyapunov residual norms can potentially be very large.

\subsection{Relaxation strategies and stopping criteria for the inner iteration}\label{relax:strat}

In order to achieve accurate results, the difference between the true and computed residual, the residual gap, needs to be small. 

% For finding relaxation strategies for the linear system residual norms, we need to estimate the Lyapunov residual gap. 
Corollary~\ref{cor:deltahybound} indicates that $ \|e_k^* H_j^{-1}Y_j\|$ decreases with the computed Lyapunov residual, and hence $\|s_k\|$ may be relaxed
during the RKSM iteration in a manner inverse proportional to the norm of the Lyapunov residual (assuming the Lyapunov residual norm decreases). 
\begin{theorem}[Theoretical relaxation strategy in RKSM]
\label{th:relax_theo}
Let the assumptions of Theorem~\ref{th:deltaybound} and Corollary~\ref{cor:deltahybound} hold.
Assume we carry out $j$ steps of Algorithm~\ref{alg:rksm} always using the explicit projection $T^{\text{expl.}}_j$. 
% Let the residual gap be given by $\eta_j=-(I-Q_jQ_j^*)S_jH_j^{-1}Y_jQ_j^*$. For a given
% $\varepsilon\in\mathbb{R}$ with $\varepsilon>0$, assume that $\|\eta_k\|<\varepsilon$ for $k<j$. 
If we choose the solve tolerances $\|s_k\|$, $1\leq k\leq j$ for the inexact
solves within RKSM such that
\begin{align}\label{bound_sk_theo}
 \|s_k\|\le\tau_k^{\text{LS}}=\begin{cases}
\tfrac{\varepsilon}{jc_A\|e_{1}^*H_j^{-1}\|\|\cR^{\text{true}}_{0}\|},&k=1,\\          
\tfrac{\varepsilon}{\frac{j}{\phi_{j}^{(k)}\|g_{k-1}\|}\|\cR^{\text{comp.}}_{k-1}\|+jc_A\|e_{k}^*H_j^{-1}\|\|\cR^{\text{true}}_{k-1}\|},&k>1,\\
           \end{cases}
\end{align}
with the same notation as before, then, for the residual gap $\|\eta_j\|\leq\varepsilon$ holds. 
\end{theorem}
\begin{proof}
Consider a single addend in the sum expression~\eqref{rksmgap:bound1} for the norm of the residual gap
\begin{align*}
 \|s_k\| \|e_k^*
H_j^{-1}Y_j\|\leq\|s_k\|\left(\frac{1}{\phi_{j}^{(k)}\|g_{k-1}\|}\|\cR^{\text{comp.}}_{k-1}\|+c_A\|e_{k}^*H_j^{-1}\|\|\cR^{\text{true}}_{\ell-1}
\|\right), k>1.
\end{align*} 
Choosing $s_k$ such that~\eqref{bound_sk_theo} is satisfied for $1\leq k\leq j$ then gives $\|\eta_j\|\leq\displaystyle\sum\limits_{k=1}^j\frac{\varepsilon}{j}=\varepsilon$
where we have used \eqref{rksmgap:bound1}, Theorem~\ref{th:deltaybound}, Corollaries~\ref{corYentrybound} and~\ref{cor:deltahybound}.
\end{proof}
The true norms $\|\cR^{\text{true}}_{k-1}\|$ can be estimated by 
$\|\cR^{\text{true}}_{k-1}\|\leq\|\cR^{\text{comp.}}_{k-1}\|+\|\eta_{k-1}\|$. For this, we might either use some estimation for $\|\eta_{k-1}\|$ or simply
assume that all previous residual gaps were sufficiently small, i.e., $\|\eta_{k-1}\|\leq \varepsilon$.
% For the
% obtained Galerkin solution $Y_j$ we can apply Theorem~\ref{th:deltaybound}, Corollary~\ref{corYentrybound} to obtain
% \begin{align}
% 	\|Y_je_k\|\leq
% \frac{(1+\sqrt{2})^2}{2\alpha_A}\|\cR^{\text{true}}_{k-1}\|\leq\frac{(1+\sqrt{2})^2}{2\alpha_A}\left(\|\cR^{\text{comp.}}_{k-1}\|+\|\eta_{k-1}\|\right).
% \end{align}
% because the explicit projection $T^{\text{expl.}}_j$ is used in~\eqref{projlyap}.

\paragraph{Practical relaxation strategies for inexact RKSM}
The relaxation strategy in Theorem~\ref{th:relax_theo} is far from practical. First, the established bounds for the entries and rows of $Y_j$ and $H_j^{-1}Y_j$, respectively, can be a vast overestimation of the true norms. Hence, the potentially large denominators
in~\eqref{bound_sk_theo} result in very small solve tolerances~$\tau_k^{\text{LS}}$ and, therefore, prevent a substantial relaxation of the inner solve
accuracies. Second, several quantities in the used bounds are unknown at step $k<j$, e.g. $H_j^{-1}$, $Y_j$, and the constants~$\phi_{j}^{(k)}$.
If we knew $H_j^{-1}$, $Y_j$, we could employ a relaxation strategy of the form $\|s_k\|\le \frac{\varepsilon}{j \|e_k^* H_j^{-1}Y_j\|}$  
 and only use Corollary~\ref{cor:deltahybound} as theoretical indication that $\|e_k^* H_j^{-1}Y_j\|$ decreases as the outer method converges.

In the following we therefore aim to develop a more practical relaxation strategy by trying to estimate the relevant quantity $\|e_k^* H_j^{-1}Y_j\|$
differently using the most recent 
available data. Suppose an approximate solution with residual norm $\|\cR^{\text{true}}_{k}\|\leq\varepsilon$,  $0<\varepsilon\ll 1$ is desired which should be
found after at most $\jmax$ rational Arnoldi steps. This goal is achieved if $\|\eta_{\jmax}\|\leq\tfrac{\varepsilon}{2}$ and if
$\|\cR^{\text{comp.}}_{\jmax}\|\leq\tfrac{\varepsilon}{2}$ is obtained by the inexact RKSM.

Consider the left null vectors of the augmented Hessenberg matrices, $\omega_{k}\underline{H}_{k}=0$, $k\leq \jmax$. It is easy to show that $\omega_{k}$ can
be updated recursively, in particular it is possible to compute $\omega_{m}=\omega_{k}(1:m+1)$, $m\leq k\leq\jmax$. Consequently, at the beginning of step $k$ 
we already have $\underline{H}_{k-1}$ and hence, also the first $k$ entries of $\omega_{\jmax}$ without knowing the full matrix
$\underline{H}_{\jmax}$.
% {\textbf{Notation, waere hier vielleicht $H_{k+1,k}$, $H_{\jmax+1,\jmax}$ besser?}}
Using Lemma~\ref{Lem:ekinvHj} and proceeding similar as in the proof of Corollary~\ref{cor:deltahybound} results in
\begin{align*}
 \|e_k^* H_{\jmax}^{-1}Y_{\jmax}\|&=\|e_k^* H_{\jmax}^{-1}(\smb Y_{k-1}&0\\0&0\sme+\Delta
Y_{{\jmax},\ell-1})\|\\
&=\left\|-\frac{\omega_{\jmax}(k)}{\phi_{\jmax}^{(k)}}[h_{k,k-1}e_{k-1}^*H_{k-1}^{-1},*]\smb Y_{k-1}&0\\0&0\sme+e_k^* H_{\jmax}^{-1}\Delta
Y_{{\jmax},\ell-1})\right\|\\
&=\left\|-\frac{h_{k,k-1}\omega_k(k)}{\phi_{\jmax}^{(k)}}e_{k-1}^*H_{k-1}^{-1}Y_{k-1}+e_k^* H_{\jmax}^{-1}\Delta Y_{{\jmax},k-1}\right\|\\
 &\leq
\left|\frac{h_{k,k-1}\omega_k(k)}{\phi_{\jmax}^{(k)}}\right|\|e_{k-1}^*H_{k-1}^{-1}Y_{k-1}\|+\|e_k^*H_{\jmax}^{-1}\|c_A\|\cR^{\text{true}}_{k-1}\|\\
 &\approx \left|\frac{h_{k,k-1}\omega_k(k)}{\phi_{\jmax}^{(k)}}\right|\|e_{k-1}^*H_{k-1}^{-1}Y_{k-1}\|,
\end{align*}
if $\|e_k^*H_{\jmax}^{-1}\|c_A\|\cR^{\text{true}}_{k-1}\|$ is small. Only the scaling parameter $\phi_{\jmax}^{(k)}$ contains missing
data at the beginning of step $k$, because $H_{k-1},~Y_{k-1}$, $\omega_k$ are known from the previous step. 
We suggest to omit the unknown data and use the following \emph{practical relaxation strategy}
\begin{subequations}\label{relax_sk_prac}
\begin{align}\label{relax_sk_prac1}
 \|s_k\|\le\tau_k^{\text{LS}}=
 \begin{cases}
\delta\frac{\varepsilon}{\jmax},&k=1,\\
% \delta\frac{\varepsilon}{\jmax \left|h_{k,k-1}\omega_k(k)\right|\|e_{k-1}^*H_{k-1}^{-1}Y_{k-1}\|},&k>1,
\delta\frac{\varepsilon}{\jmax \|h_{k,k-1}e_{k-1}^*H_{k-1}^{-1}Y_{k-1}\|},&k>1,
\end{cases}
\end{align}
where $0<\delta\leq1$ is a safeguard constant to accommodate for the estimation error resulting from approximating $\|e_k^* H_{\jmax}^{-1}Y_{\jmax}\|$ and
omitting
unknown quantities (e.g., $\phi_{\jmax}^{(k)}$). 
% In our experiments, $\delta=0.1/\|BB^*\|$ was sufficient. 

The reader might notice by following Algorithm~\ref{alg:rksm} closely, that the built up subspace at the beginning of iteration step $k\leq \jmax$ is already
$k$-dimensional and
since we are using the explicit projection
$T^{\text{expl.}}_k$ to define the Galerkin systems, we can already compute $Y_k$ directly after building $Q_k$. This amounts to a simple rearrangement of
Algorithm~\ref{alg:rksm} by moving Lines~\ref{rksmAj},~\ref{solveY} before Line~\ref{rksmLGS}.
Hence, the slight variation 
\begin{align*}
 \|e_k^* H_{\jmax}^{-1}Y_{\jmax}\|\approx\left|\frac{\omega_{k}(k)}{\phi_{\jmax}^{(k)}}\right|\|[-h_{k,k-1}e_{k-1}^*H_{k-1}^{-1},1]Y_k\|
\end{align*}
of the above estimate is obtained, which suggests the use of the (slightly different) \emph{practical relaxation strategy}
\begin{align}\label{relax_sk_prac2}
 \|s_k\|\le\tau_k^{\text{LS}}=
 \begin{cases}
\delta\frac{\varepsilon}{\jmax},&k=1,\\
% \delta\frac{\varepsilon}{\jmax \left|\omega_k(k)\right|\|[-h_{k,k-1}e_{k-1}^*H_{k-1}^{-1},1]Y_k\|},&k>1.
\delta\frac{\varepsilon}{\jmax \|[-h_{k,k-1}e_{k-1}^*H_{k-1}^{-1},1]Y_k\|},&k>1.
\end{cases}
\end{align}
\end{subequations}

For both dynamic stopping criteria, in order to prevent too inaccurate and too accurate linear solves, it is reasonable to enforce
$\tau_k^{\text{LS}}\in[\tau^{\text{LS}}_{\min},\tau^{\text{LS}}_{\max}]$, where $0<\tau^{\text{LS}}_{\min}<\tau^{\text{LS}}_{\max}\leq 1$ indicate minimal and
maximal linear solve thresholds. 

The numerical examples in Section \ref{sec:num} show that these practical relaxation strategies are effective and can reduce the number of inner iterations for RKSM by up to 50 per cent.

\subsection{Implementation issues and generalizations}\label{ssec:imple_gen}

This section contains several remarks on the implementation of the inexact RKSM, in particular the case when the right hand side of the Lyapunov equation has
rank greater than one, as well as considerations of the inner iterative solver and preconditioning. We also briefly comment on extensions to generalized
Lyapunov equations and algebraic Riccati equations.

\paragraph{The case $r>1$}
The previous analysis was restricted to the case $r=1$ but  the block situation, $r>1$, can be handled similarly by using appropriate block versions of the 
 relevant quantities, e.g., $q_k,~w_k,~s_k\in\Cnr$, $h_{ij}\in\Crr$, $\omega\in\C^{r\times (j+1)r}$, and $e_k$ by $e_k\otimes I_r$, as well as replacing
absolute values by spectral norms in the right places. When solving the linear system with $r$ right hand sides $q_k$, such that, $\|s_k\|\leq\tau_k^{\text{LS}}$
is achieved, one can either used block variants of the iterative methods (see, e.g.,~\cite{Sood15}), or simply apply a single vector method and sequentially consider every column
$q_k(:,\ell)$, $\ell=1,\ldots,r$ and demand that $s_k(:,\ell)\leq\tau_k^{\text{LS}}/r$.
% \end{remark}
%%%%%%%%%%%%%%%%%%%%%%%%
\paragraph{Choice of inner solver}
One purpose of low-rank solvers for large matrix equations is to work in a memory efficient way. Using a long-recurrence method such as GMRES to solve unsymmetric inner linear systems defies this purpose in some sense because it requires storing the full Krylov basis.
Unless a very good preconditioner is available, this can lead to significant additional storage requirements within the inexact low-rank method, where the
Krylov method consumes more memory than the actual low-rank Lyapunov solution factor of interest. Therefore we exclusively used short-recurrence Krylov methods
(e.g., BiCGstab) for the numerical examples defined by an unsymmetric matrix $A$.
% \end{remark} 
% {\textbf{Nur ne idee, nicht fuers paper: Kann man da was mit recycling von previous iterations machen? Im prinzip ein Punkt den man weiter verfolgen könnte,
% auch wenn dabei die Krylov basis des inneren Lösers aufgehoben werden muss. Sowas in der Art wurde in~\cite{morLi00} erwähnt.}}
%%%%%%%%%%%%%%%%%%%%%%%%
\paragraph{Preconditioning}
The preceding relaxation strategies relate to the residuals $s_k$ of the underlying linear systems. For enhancing the performance of the Krylov subspace methods, using preconditioners is typically inevitable. Then the inner iteration itself inherently only works with the preconditioned residuals $s^{\text{Prec.}}_k$ which, if left or two-sided preconditioning is used, are different from the residuals $s_k$ of the original linear systems. Since $\|s^{\text{Prec.}}_k\|\leq \tau_k^{\text{LS}}$ does not imply $\|s_k\|\leq \tau_k^{\text{LS}}$ this can result in linear systems solved not accurately enough to ensure small enough Lyapunov residual gaps.
Hence, some additional care is needed to respond to these effects from preconditioning. The obvious approach is to use right preconditioning which gives
$\|s^{\text{Prec.}}_k\|=\|s_k\|$.
% \end{remark}
%%%%%%%%%%%%%%%%%%%%%%%%
\paragraph{Complex shifts}
In practice $A$, $B$ are usually real but some of the shift parameters can occur in complex conjugate pairs. 
Then it is advised to reduce the amount of complex operations by working with a real rational Arnoldi method~\cite{Ruh94c} that constructs a real rational
Arnoldi decomposition and slightly modified formulae for the computed Lyapunov residuals, in particular for $F_j$. 
The actual derivations are tedious and are omitted here for the sake of brevity, but our implementation for the numerical experiments works exclusively with the real Arnoldi process.  
%\textbf{Diesbezüglich haut in der Implementierung etwas noch nicht ganz hin (es geht nur so halbwegs)... Ich bin da aber dran. 
%Der Remark sollter aber ausreichen, denke ich.
%}
% \end{remark}
%%%%%%%%%%%%%%%%%%%%%%%%
\paragraph{Generalized Lyapunov equations}
Often, generalized Lyapunov equations of the form
\begin{align}\label{glyap}
 AXM^*+MXA^*=-BB^*,
\end{align}
with an additional, nonsingular matrix $M\in\Rnn$ have to be solved.
Projection methods tackle~\eqref{glyap} by implicitly running on equivalent Lyapunov equations defined by 
$A_M:=L_M^{-1}AU_M^{-1}$, $B_M:=L_M^{-1}B$ using a factorization $M=L_MU_M$, which could be a LU-factorization or, if $M$ is
positive definite, a Cholesky factorization ($U_M=L_M^*$). Other
possibilities are $L_M=M,~U_M=I$ and $L_M=I,~U_M=M$. 
Basis vectors for the projection subspace are obtained by 
\begin{align}
 Q_1\beta&=L_M^{-1}B,\quad (A-\xi_{j+1}M)\hw_{j+1}=L_Mq_{j},\quad w_{j+1}=U_M\hw_{j+1}.
\end{align}
After convergence, the low-rank approximate
solution of the original problem~\eqref{glyap} is given by $X\approx (U_M^{-1}Q_j)Y_j(U_M^{-1}Q_j)^*)$, where $Y_j$ solves the reduced Lyapunov equation
defined by the restrictions of $A_M$ and $B_M$.
This requires solving extra linear systems defined by (factors of) $M$ in certain stages of Algorithm~\ref{alg:rksm}: setting up the
first
basis vector, building the restriction $T_j$ of $A_M$ (either explicitly or implicitly using~\eqref{exactProj}), and recovering the approximate solution after
termination.  
Since the coefficients of these linear system do not change throughout the iteration, often sparse factorizations of $M$ are computed once at the beginning
and
reused every time they are needed. In this case the above analysis can be applied with 
minor changes of the form
\begin{align*}
s_j:=L_Mq_j-(A-\xi_{j+1} M)\hat\tw_{j+1},\quad \|s_j\|\leq \tau^{\text{LS}}_j,\quad w_{j+1}=U_M\hat\tw_{j+1}, 
\end{align*}
for the inexact linear solves.  We obtain an inexact rational Arnoldi decomposition with respect to $A_M$ of the form
\begin{align*}
 A_MQ_j&=Q_j\hT^{\text{expl.}}_j+\hg_je_j^*h_{j+1,j}H_j^{-1}-(I-Q_jQ_j^*)L_M^{-1}S_jH_j^{-1},\\
 \text{with}\quad\hT^{\text{expl.}}_j&=Q_j^*A_MQ_j,\quad \hg_j=q_{j+1}\xi_{j+1}-(I-Q_jQ_j^*)A_M q_{j+1}.
\end{align*}
If $Q_jY_jQ_j^*$ is an approximate solution of the equivalent Lyapunov equation defined by $A_M,~B_M$, and $Y_j$ solves the reduced Lyapunov equation
defined by $\hT^{\text{expl.}}_j$ and $Q_j^*B_M$, then the associated residual is
\begin{align*}
 A_MQ_jY_jQ_j^*+Q_jY_jQ_j^*A_M^*+B_MB_M^*&=\hF_j+\Delta \hF_j+(\hF_j+\Delta \hF_j)^*,\\
 \hF_j:=\hg_j^*h_{j+1,j}H_j^{-1}Y_jQ_j^*,\quad \Delta \hF_j&:=(I-Q_jQ_j^*)L_M^{-1}S_jH_j^{-1}Y_jQ_j^*.
\end{align*}
Hence, the generalized residual gap is $\hat\eta_j=\Delta \hF_j$. If $L_M=I$, $U_M=M$, bounding $\|\hat\eta_j\|$ works in the same way as in the case $M=I$,
otherwise an additional constant $1/\sigma_{\min}(L_M)$ (or an estimation thereof) has to be multiplied to the established bounds.
Allowing inexact solves of the linear systems defined by (factors of) $M$ substantially complicates the analysis. In particular, the transformation to a
standard Lyapunov equation is essentially not given exactly since, in that case, only a perturbed version of $A_M$ and its restriction are available. This
situation is, hence, similar to the case when no exact matrix vector products
with $A$ are available. If $L_M\neq I$, also $B_M$ is not available exactly leading to further perturbations in
the basis generation. 
For these reasons, we will not further pursue inexact solves with $M$ or its factors. This is also motivated from practical situations,
where solving with $M$, or computing a sparse factorization thereof, is usually much less demanding compared to factorising $A-\xi_{j+1} M$.
%%%%%%%%%%%%%%%%%%%%%%%%%%%%%%%%%%%%%%%%%%%%%%%%%%%%%%%%%%%%%%%%%%%%%%%%%%%%%%%%
\paragraph{Extended Krylov subspace methods}
A special case of the rational Krylov subspace appears when only the shifts zero and infinity are used, leading to the extended Krylov subspace 
$\cE\cK_k(A_M,B_M)=\cK_k(A_M,B_M)\cup\cK_k(A^{-1}_M,A^{-1}_MB_M)$ (using the notation from the previous subsection). Usually, in the resulting extended Arnoldi process the basis is expanded by vectors from
$\cK_k(A_M,B_M)$ and $\cK_k(A^{-1}_M,A^{-1}_MB_M)$ in an alternating fashion, starting with $\cK_k(A_M,B_M)$.
The extended Krylov subspace method (EKSM)~\cite{Sim07} for~\eqref{lyap} and~\eqref{glyap} uses a Galerkin projection onto $\cE\cK_k(A_M,B_M)$.
In each step the basis is orthogonally expanded by $w_{j+1}=[A_Mq_j(:,1:r),A_M^{-1}q_j(:,r+1:2r)]$, where $q_j$ contains the last $2r$ basis vectors. This
translates to the following linear systems and matrix vector products
\begin{align*}
 U_Mz&=q_j(:,1:r),\quad L_Mw_{j+1}(:,1:r)=Az,\\
\quad\text{and}\quad Az=&L_Mq_j(:,r+1:2r),\quad w_{j+1}(:,r+1:2r)=U_Mz, 
\end{align*}
that have to be dealt with. Similar formula for the implicit restriction of $A$ and the Lyapunov residual as in RKSM can be found in~\cite{Sim07}.
Since these coefficient matrices do not change during the iteration, a very efficient strategy is to compute, if possible, sparse factorizations of
$A,M$ once before the algorithm and reuse them in every step. Incorporating inexact linear solves by using inexact
sparse
factorizations $A\approx \tL_A\tU_A$, $M\approx \tL_M\tU_M$ would make it difficult
to incorporate relaxed solve tolerances, since there is little reason to compute a less accurate factorization once a highly accurate one has been constructed.
For the same reasons stated in the paragraph above, we restrict ourselves to the iterative solution of the linear systems defined by $A$. These
linear systems affect only the columns in $w_{j+1}(:,r+1:2r)$. In particular, by proceeding as for inexact RKSM, one can show that
\begin{align*}
 A_MQ_j&=Q_{j+1}T_{j+1,j}^{\text{expl.}}-(I-Q_jQ_j^*)S_j^{\text{EK}}\quad\text{with}\\
 T_{j+1,j}^{\text{expl.}}&=Q_{j+1}^*A_MQ_j,\quad S_j^{\text{EK}}:=[s_1^{\text{EK}},\ldots,s_j^{\text{EK}}],\quad
s_i^{\text{EK}}:=[0,U_M^{-1}s_i]\in\C^{n\times 2r}, 
\end{align*}
where $s_i:=L_Mq_i(:,r+1:2r)-Az$, $1\leq i\leq j$. Note that, allowing inexact solves w.r.t. $M$ or even inexact matrix vector products with $A,~M$, would
destroy the zero block columns in $s_i^{\text{EK}}$.
% {\textbf{Hier koennte ein referee vielleicht ein Bsp verlangen, da wir auch fuer ARE und gen Lyap eins haben ..., ich hoffe nicht. müsste ich erst noch
% implementieren}}
%%%%%%%%%%%%%%%%%%%%%%%%%%%%%%%%%%%%%%%%%%%%%%%%%%%%%%%%%%%%%%%%%%%%%%%%%%%%%%%%
\paragraph{Algebraic Riccati equations}
The RKSM in Algorithm~\ref{alg:rksm} can be generalized to compute low-rank approximate solutions of generalized algebraic Riccati equations (AREs)
\begin{align}\label{care}
 AXM^*+MXA^*-MXCC^*XM+BB^*=0,\quad C\in\Cnp,\quad p\ll n,
\end{align}
see, e.g.,~\cite{SimSM14,Sim16}. 
The majority of steps in Algorithm~\ref{alg:rksm} remain unchanged, the main difference is that the Galerkin system is now a reduced ARE 
\begin{align*}
 T_jY_j+Y_jT_j^*-Y_j(Q_j^*C_M)(C_M^*Q_j)Y_j+Q_j^*B_MB_M^*Q_j=0,\quad C_M:=U_M^{-*}C,
\end{align*}
which has to be solved. Since we do not alter the underlying rational Arnoldi process, most of the properties of the RKSM hold again, especially the
residual formulas in both the exact and inexact case, and a residual gap is defined again as in the Lyapunov case. Differences occur in the bounds
for the entries of $Y_j$ and rows of $H_j^{-1}Y_j$, since Theorem~\ref{th:deltaybound} cannot be formulated in the same way. Under some additional assumptions,
a bound of the form $\|\Delta Y_{j,j-1}\|\leq c_{\text{ARE}}\|\cR_{j-1}\|$ can be established~\cite{Sim16}, where the constant $c_{\text{ARE}}$ is different
from $c_A$. 
% Numerical observation nevertheless indicate that 
We leave concrete generalizations of Theorem~\ref{th:deltaybound}, Corollaries~\ref{corYentrybound},~\ref{cor:deltahybound} for future research and only show
in some experiments that relaxation strategies of the form~\eqref{relax_sk_prac1},~\eqref{relax_sk_prac2} also work for the inexact RKSM for AREs.

\section{The inexact low-rank ADI iteration}
\label{sec:adi}

\subsection{Derivation, basic properties, and introduction of inexact solves}
Using the Cayley transformation $\cC(A,\alpha):=(A+\alpha I)^{-1}(A-\overline{\alpha}
I)$, for $\alpha\in\C_-$,
%\quad\text{and}\\\label{adi_rhs}
 %\cT(\alpha)&:=-2\Real{\alpha}(A+\alpha I_n)^{-1}BB^*(A+\alpha I_n)^{-H}.
%\end{align}
~\eqref{lyap} can be reformulated as discrete-time Lyapunov equation (symmetric Stein equation)
\begin{align*}
	X=\cC(A,\alpha)X\cC(A,\alpha)^*-2\Real{\alpha}\cB(\alpha)\cB(\alpha)^*,~\cB(\alpha):=(A+\alpha I)^{-1}B.
\end{align*}
For suitably chosen $\alpha_j$ this motivates the iteration
\begin{align}\label{adi_iter1}
\begin{split}
 X_{j}
%&=(A+\alpha_jI_n)^{-1}(A-\overline{\alpha_j}
%I_n)X_{j-1}(A-\overline{\alpha_j} I_n)^*(A+\alpha_j I_n)^{-H}\\
%&-2\Real{\alpha_j}(A+\alpha_j I_n)^{-1}BB^*(A+\alpha_j I_n)^{-H}\\
&=\cC(A,\alpha_j)X_{j-1}\cC(A,\alpha_j)^*-2\Real{\alpha_j}\cB(\alpha_j)\cB(\alpha_j)^*,
\end{split}
\end{align}
which is known as alternating directions implicit (ADI) iteration~\cite{Wac13} for Lyapunov equations. It converges for shift parameters $\alpha_j\in\C_-$ because $\rho(\cC(A,\alpha_j))<1$ and 
%\end{subequations}
it holds~\cite{LiW02,Sab07,Wac13,morBenKS13}
\begin{align}\label{adi_error}
 X_j-X&=\cA_j(X_0-X)\cA_j^*,\\\label{lradi_res}
    \cR_j&=A X_j+X_jA^*+BB^*=\cA_j\cR_0\cA_j^*,
\end{align}
where $\cA_j:=\prod\limits_{i=1}^{j}\cC(A,\alpha_i)$.

A low-rank version of the ADI iteration is obtained by setting $X_0=0$ in~\eqref{adi_iter1}, exploiting that $(A+\alpha_j I)^{-1}$ and $(A-\overline{\alpha_i}
I)$ commute for $i,j\geq 1$, and realising that the iterates are given in low-rank factored form $X_j=Z_jZ_j^*$ with low-rank factors $Z_j$ constructed by
% \begin{align}\label{pre_lradi}
% %\begin{split}
%   Z_1=\gamma_1\cB(\alpha_1),\quad
% Z_j=\left[\gamma_j\cB(\alpha_j),\quad
% \cC(A,\alpha_j)Z_{j-1}\right],~j>1,\quad \gamma_j:=\sqrt{-2\Real{\alpha_j}}.
% %\end{split}
% \end{align}
% Because $(A+\alpha_j I)^{-1}$ and $(A-\overline{\alpha_i} I)$ commute for $i,j\geq 1$,
% the generation of $Z_j$ can, after reversing the order of the shifts, be expressed as
%where we exploited the symmetric structure of $\cT(\alpha_j)$. In this way, $rj$
%columns have to be processed at iteration $j$, i.e., $j$ shifted linear system
%with $r$ right hand
%sides each have to be solved, such that the iteration gets increasingly
%expensive. Introducing the notations 
%\begin{align*}
 %\gamma_j&:=\sqrt{-2\Real{\alpha_j}},\quad \cB_j:=(A+\alpha_j I_n)^{-1}B,\quad \cC_j:=\cC(A,\alpha_j),
%\end{align*}
 %the low-rank factor $Z_j$ in the above scheme can be written as
%\begin{align*}
 %Z_j=\left[\gamma_j\cB_j,~\gamma_{j-1}\cC_j\cB_{j-1},~\gamma_{j-2}\cC_j\cC_{j-1}
%\cB_{j-2},\ldots,\gamma_1\cC_j\cdots\cC_1\cB_1\right].
%\end{align*}
%The Cayley transforms can be written as $\cC_i=P_iT_i$, $P_i:=(A+\alpha_i I)^{-1}$, $T_i:=(A-\overline{\alpha_i} I)$ with the commutative properties
%$T_iT_j=T_jT_i$, $P_iP_j=P_jP_i$, $P_iT_j=T_jP_i$ $\forall i,j\geq 1$. Therefore, reversing the order of the shifts does no harm, and the low-rank factor is
%given by
\begin{align}\label{pre_lradi}
\begin{split}
  Z_j&=\left[\gamma_1v_1,~\gamma_2v_2,~\ldots,\gamma_jv_j\right]
 =\left[Z_{j-1},\gamma_jv_j\right],\quad \gamma_j:=\sqrt{-2\Real{\alpha_j}},\\ 
v_j&=(A-\overline{\alpha_{j-1}}I)(A+\alpha_jI)^{-1}v_{j-1},~j\geq 1,\quad v_1:=(A+\alpha_1 I)^{-1}B,
\end{split}
\end{align}
see~\cite{LiW02,Saa09} for more detailed derivations. Thus, in each step $Z_{j-1}$ is augmented by $r$ new columns $v_j$. Moreover, from~\eqref{lradi_res} with
$X_0=0$ and~\eqref{pre_lradi} it is evident that
\begin{align}\label{adi_resexact}
 \cR_j=w_jw_j^*,\quad w_j:=\cA_jB=w_{j-1}-\gamma^2_j(A+\alpha_jI)^{-1}w_{j-1},\quad w_0:=B,
\end{align}
 see also~\cite{morBenKS13,WolP16}. Hence, the residual matrix has at most rank $r$ and its norm can be cheaply computed as $\|R_j\|_2=\|w_j^*w_j\|_2$
which coined the name \textit{residual factors} for the $w_j$.
The low-rank ADI (LR-ADI) iteration using these residual factors is
\begin{subequations}\label{lradi_basic}
 \begin{align}\label{lradi_I}
	v_j	=(A+\alpha_j I)^{-1}w_{j-1},\quad w_j=w_{j-1}+\gamma_j^2v_j,\quad w_0:=B.
\end{align}

For generalized Lyapunov equations~\eqref{glyap}, this iteration can be formally applied to an equivalent Lyapunov equations defined, e.g. by~$M^{-1}A$, $M^{-1}B$. Basic manipulations~(see, e.g.~\cite{morBenKS13,Kue16}) resolving the inverses of $M$ yield the generalized LR-ADI iteration

\begin{align}\label{lradi_M}
	v_j	=(A+\alpha_j M)^{-1}w_{j-1},\quad w_j=w_{j-1}+\gamma_j^2Mv_j,\quad w_0:=B.
\end{align}
\end{subequations} 
As for RKSM, the choice of shift parameters $\alpha_j$ is crucial for a rapid convergence. Many approaches have been developed for this problem, see, e.g., \cite{Sab07,TruV09,Wac13}, where shifts are typically obtained from (approximate) eigenvalues of $A,M$. Among those exist asymptotically optimal shifts for certain situations, e.g., $A=A^*$, $M=M^*$ and $\Lambda(A,M)\subset\R_-$. Adaptive shift generation approaches, where shifts are computed during the running iteration, were proposed in~\cite{BenKS14,Kue16} and often yield better results, especially for nonsymmetric coefficients with complex spectra. In this work we mainly work with these dynamic shift selection techniques.

The main computational effort in each step is obviously the solution of the shifted linear system with
$(A+\alpha_jM)$ and $r$ right hand sides for $v_j$ in \eqref{lradi_basic}. Allowing inexact linear solves but keeping the other steps in~\eqref{lradi_basic}
unchanged results in the \emph{inexact low-rank ADI iteration} illustrated in Algorithm~\ref{alg:lradi}. We point out that a different notion of an inexact ADI
iteration can be found
in~\cite{ReiW12} in the context of operator Lyapunov equations, where inexactness refers to the approximation of infinite dimensional operators by finite
dimensional ones.
\begin{algorithm2e}[t]
\SetEndCharOfAlgoLine{}
\SetKwInOut{Input}{Input}\SetKwInOut{Output}{Output}
  \caption{Inexact LR-ADI iteration.}
  \label{alg:lradi}
%   \begin{algorithmic}[1]
    \Input{Matrices $A,~M,~B$ defining \eqref{glyap}, set of
    shift parameters  $\lbrace
\alpha_1,\dots,\alpha_{j_{\max}}\rbrace\subset\C_-$, tolerance $0<\varepsilon\ll1$.}
    \Output{$Z_{j}\in\C^{n\times rj}$, such that
    $ZZ^*\approx X$.}
  $w_0=B,\quad Z_0=[~],\quad \cR_0 = \|w_0^*w_0\|,\quad j=1$.\;
  \While{$\|\cR_{j-1}\|\geq\varepsilon$}{%
    Get $v_j$ s.t. $s_j=w_{j-1}-(A+\alpha_jM)v_j$, $\|s_j\|\leq \delta_j$.\;
      $w_j=w_{j-1}-2\Real{\alpha_j}Mv_j$.\;% \quad
$Z_j=[Z_{j-1},~\sqrt{-2\Real{\alpha_j}}v_j]$.\;
$\cR_j = \|w_j^*w_j\|_2$.\;
$j=j+1$.\;
  }
\end{algorithm2e}

Of course, allowing $s_j\neq 0$ will violate some of the properties that were used to derive the LR-ADI iteration. 
Inexact solves within the closely related Smith methods have
been
investigated in, e.g.,~\cite{Sab07,Sun08}, from the viewpoint of an inexact nonstationary iteration~\eqref{adi_iter1} that led to rather conservative results on
the allowed magnitude of the norm of the linear system residual $s_j$. 
The analysis we present here follows a different path by exploiting the well-known connection of the
LR-ADI iteration to rational Krylov subspaces~\cite{morLi00,LiW02,DruKS11}.
\begin{theorem}[\cite{morLi00,LiW02,WolP16}]
The low-rank solution factors $Z_j$ after $j$ steps of the exact LR-ADI iteration ($\|s_j\|=0,~\forall j\geq1$) span a (block) rational Krylov subspace:
\begin{align}
	\range{Z_j}\subseteq\range{\left[(A+\alpha_1 M)^{-1}B,\ldots,\left[\prod\limits_{i=2}^j(A+\alpha_iM)^{-1}\right](A+\alpha_1 M)^{-1}B\right]}.
\end{align}
\end{theorem}
Although for LR-ADI, the $Z_j$ do not have orthonormal columns and there is no rational Arnoldi process in Algorithm~\ref{alg:lradi}, it is still possible to find
decompositions similar to~\eqref{exactRatArnDecomp} and \eqref{inexactRatArnDecomp1}.
\begin{theorem}\label{thm:ADI_ratarndec}
The low-rank solution factors $Z_j$ after $j$ steps of the inexact LR-ADI iteration (Algorithm~\ref{alg:lradi}) satisfy a rational Arnoldi like decomposition
\begin{align}\label{ADI:ratarndecomp}
	AZ_j&=MZ_jT_j+w_jg_j^*-S_j\Gamma_j,\quad S_j:=[s_1,\ldots,s_j],\quad \text{where}\\
	T_j&=-\smb\alpha_1&0&\ldots&0\\
	\gamma_1\gamma_2&\alpha_2&\ddots&\vdots\\
	\vdots&&\ddots&0\\
	\gamma_1\gamma_j&\ldots&&\alpha_j
	\sme\otimes I_r,\quad
	g_j:=\smb \gamma_1\\\vdots\\\gamma_j\sme\otimes I_r,\quad\Gamma_j:=\diag{g_j}.
\end{align}
Moreover, the Lyapunov residual matrix associated with the approximation $X_j= Z_jZ_j^*\approx X$ is given by
\begin{align}\label{ADI:res_inex}
	\cR_j^{\text{true}}=AZ_jZ_j^*M^*+MZ_jZ_j^*A^*+BB^*=-S_j\Gamma_jZ_j^*M^*-MZ_j\Gamma_jS_j^*+w_jw_j^*.
\end{align}
\end{theorem}
Note that here, $T_j$ and $g_j$ denote different quantities than in Section \ref{sec:rksm}.
\begin{proof}
For $S_j=0$ the decomposition~\eqref{ADI:ratarndecomp} was established in~\cite{Kue16,WolP16} and can be entirely derived from the
relations~\eqref{lradi_basic}. 
Inexact solves in the sense $w_{j-1}-(A+\alpha_jM)v_j=s_j$ can be inserted in a straightforward way leading to~\eqref{ADI:ratarndecomp}. 
By construction, it holds $B=w_0=w_j-MZ_jg_j$ which, together with~\eqref{ADI:ratarndecomp}, gives
\begin{align*}
	\cR_j^{\text{true}}
	%&=(MZ_jT_j+w_jg_j^*-S_j\Gamma_j)Z_j^*M^*+MZ_j(MZ_jT_j+w_jg_j^*-S_j\Gamma_j)^*\\
	%&+(w_j-MZ_jg_j)(w_j-MZ_jg_j)^*\\
	&=MZ_j(T_j+T_j^*+g_jg_j^*)Z_j^*M^*+w_jw_j^*-S_j\Gamma_jZ_j^*M^*-MZ_j\Gamma_jS_j^*
\end{align*}
and~\eqref{ADI:res_inex} follows by verifying that $T_j+T_j^*=-g_jg_j^*$.
\end{proof}
\begin{remark}
The LR-ADI iteration is in general not a typical Galerkin projection method using orthonormal bases of the search spaces and solutions of reduced matrix
equations.  For instance, in~\cite{WolP16} it is shown that the exact LR-ADI iteration can be seen as an implicit Petrov-Galerkin type method with a hidden left
projection subspace. In particular, as we exploited in the proof above,
the relation $T_j+T_j^*+g_jg_j^*=0$ can be interpreted as reduced matrix equation solved by the identity matrix $I_j$.
It is also possible to state a decomposition similar to~\eqref{ADI:ratarndecomp} which incorporates $w_0=B$ instead of $w_j$~\cite{DruKS11,WolP16}. This can be
used to state conditions which indicate when the LR-ADI approximation satisfies a Galerkin condition~\cite[Theorem~3.4]{DruKS11}.
Similar to inexact RKSM, if $s_j\neq 0$ these result do in general not hold any more.
\end{remark}
%%%%%%%%%%%%%%%%%%%%%%%%%%%%%%%%%%%%%%%%%%%%%%%
\subsection{Computed Lyapunov residuals, residual gap, and relaxation strategies in inexact LR-ADI}
Similar to inexact RKSM when inexact solves are allowed in LR-ADI, the computed Lyapunov residuals $\cR^{\text{comp.}}_j = w_jw_j^*$ are different
from the true residuals $\cR_j^{\text{true}}$ and, thus, $\|w_j\|^2$ ceases to be a safe way to assess the quality of the current approximation $Z_jZ_j^*$.
In the exact case, $s_j=0$, it follows from~\eqref{adi_resexact} and $\rho_j:=\rho(\cC_j)<1$ that the Lyapunov residual norms decrease in the form
$\|\cR_j\|\leq
c\rho_j^2\|\cR_{j-1}\|$ for some $c\geq 1$. 
For the factors $w_j$ of the computed Lyapunov residuals $\cR^{\text{comp.}}_j$ in the inexact method we have the following result.
\begin{lemma}\label{Lem:adi_compres}
Let  $w^{\text{exact}}_{j}$ be the factors of the Lyapunov residuals of the exact LR-ADI method, i.e., $s_i=0,\,1\leq i\leq j$. The factors $w_j$ of the computed
Lyapunov residuals of the inexact LR-ADI are given by
 \begin{align*}
w_j=w^{\text{exact}}_{j}+\sum\limits_{i=1}^j\left[\prod\limits_{k=i+1}^j\cC_k\right](\cC_i-I) s_i.
\end{align*}
\end{lemma}
\begin{proof}
For simplicity, let $r=1$, $M=I_n$, and exploit $\cC_j:=\cC(A,\alpha_j)=A_{-\overline{\alpha_j}}A_{\alpha_j}^{-1}=I+\gamma_j^2A_{\alpha_j}^{-1}$,
$A_{\alpha}=A+\alpha I$. Denoting the errors by $\delta
v_j=A_{\alpha_j}^{-1}w_{j-1}-v_j=A_{\alpha_j}^{-1}s_j$ for $j\geq 1$, we
have
\begin{align*}
 w_j&=w_{j-1}+\gamma_j^2 v_j=w_{j-1}+\gamma_j^2 A_{\alpha_j}^{-1}w_{j-1}-\gamma_j^2\delta v_j=\cC_jw_{j-1}-\gamma_j^2\delta v_j\\
&=\ldots=\prod\limits_{k=1}^j\cC_k w_0-\sum\limits_{i=1}^j\gamma_i^2\left[\prod\limits_{k=i+1}^j\cC_k\right]\delta v_i=\prod\limits_{k=1}^j\cC_k w_0-\sum\limits_{i=1}^j\left[\prod\limits_{k=i+1}^j\cC_k\right](\cC_i-I) s_i
\end{align*}
and we notice that the first term is exactly the exact Lyapunov residual factor from~\eqref{adi_resexact}.
\end{proof} 
Constructing $w_jw_j^*$ from the above formula and taking norms indicates that, by the contraction property of $\cC_i$, the linear system  residuals $s_i$ get damped. In
fact, similar to the inexact projection methods, in practice we often observe that the computed Lyapunov residual norms $\|\cR^{\text{comp.}}_j\|=\|w_jw_j^*\|$
also show a decreasing behavior. 

%{\textbf{Der folgende Abschnitt ist noch eine Baustelle... das fehlt noch was}}

%For the transition from $w_j$ to $w_{j-1}$ we have
%\begin{align*}
% \omega_j^2&=\|w_j^*w_j\|=\|\cC_jw_{j-1}\|^2+\|(\cC_j-I)s_j\|^2-2\Real{w_{j-1}^*\cC_j^*(\cC_j-I)s_j}\\
% &\leq c^2\rho^2_j\omega^2_{j-1}+(\rho_j+1)^2c^2\|s_j\|^2+2c(\rho_j+1)\rho_j\omega_{j-1}\|s_j\|
%\end{align*}
%with $\|\cC_j\|\leq c\rho_j,\quad c\geq 1$.

In the next step we analyze the difference between the computed and true residuals, in a similar manner as we did for the RKSM in
Section~\ref{sec:rksm}.
%%%%%%%%%%%%%%%%%%%%%%%%%%%%%%%%%%%%%%%%%%%%%%%%%%%%%%%%%%%%%%%%%%
% \subsection{}
Theorem~\ref{thm:ADI_ratarndec} motivates the definition of a \emph{residual gap} analogue to the inexact RKSM.
\begin{defin}\label{def:resgapADI}
The \emph{residual gap} after $j$ steps of the inexact LR-ADI iteration is given by
\begin{align}\label{adi:resgap}
	\Delta\cR_j^{\text{ADI}}:=\cR_j^{\text{true}}-\cR^{\text{comp.}}_j =
\cR_j^{\text{true}}-w_jw_j^*=\eta_j^{\text{ADI}}+(\eta_j^{\text{ADI}})^*,\quad \eta_j^{\text{ADI}}:=-S_j\Gamma_jZ_j^*M^*.
\end{align}
\end{defin}

Assuming we have $\|\cR^{\text{comp.}}_j\|=\|w_jw_j^*\|\leq \varepsilon$ and are able to bound the residual gap,
e.g., $\|\Delta\cR_j^{\text{ADI}}\|\leq \varepsilon$, then we achieve small true residual norms $\|\cR_j^{\text{true}}\|\leq 2\varepsilon$. A theoretical
approach for bounding $\|\Delta\cR_j^{\text{ADI}}\|$ is given in the next
theorem.
\begin{theorem}[Theoretical relaxation for inexact LR-ADI]\label{thm:ADI_theorelax}
Let the residual gap be given by Definition~\ref{def:resgapADI} with $w_j$, $\gamma_j$ as in
\eqref{pre_lradi}-\eqref{lradi_M}. Let $\jmax$ be the maximum number of steps of Algorithm~\ref{alg:lradi}, 
%$\sigma_{\min,k}:=\sigma_{\min}(A+\alpha_kM)$,
$\sigma_{k}:=\|M(A+\alpha_kM)^{-1}\|$,
and $0<\varepsilon<1$ the desired residual
tolerance. 
\begin{subequations}\label{ADI_relax_theo}
 \begin{enumerate}
\item[(a)]
If, for $1\leq k\leq j_{\max}$, the linear system residual satisfies
\begin{align}\label{ADI_relax1_theo}
	%\|s_k\|\leq \half\left(\sqrt{\|w_{k-1}\|^2+\frac{2\varepsilon\sigma_{\min,k}}{\|M\|\gamma_k^2j_{\max}}}-\|w_{k-1}\|\right),
	\|s_k\|\leq \half\left(\sqrt{\|w_{k-1}\|^2+\frac{2\varepsilon}{\sigma_{k}\gamma_k^2j_{\max}}}-\|w_{k-1}\|\right),
\end{align} 
then 
% each summand in~\eqref{ADIresgapbound} is bounded by $\varepsilon/{j_{\max}}$ and, thus, 
$\|\Delta\cR_{\jmax}\|\leq\varepsilon$. 
\item[(b)] %Let $\|S_{k-1}\Gamma_{k-1}Z_{k-1}^*M^*\|\leq u_{k-1}$ with $u_0=0$. 
If, for $1\leq k\leq j_{\max}$, the linear system residual satisfies 
\begin{align}\label{ADI_relax2_theo}
	\|s_k\|\leq
%\half\left(\sqrt{\|w_{k-1}\|^2+2\left(\frac{k\varepsilon}{j_{\max}}-2\|\eta^{\text{ADI}}_{k-1}\|\right)\frac{\sigma_{\min,k}}{\|M\|\gamma_k^2}}-\|w_{k-1}\|\right),
\half\left(\sqrt{\|w_{k-1}\|^2+\left(\frac{k\varepsilon}{j_{\max}}-2\|\eta^{\text{ADI}}_{k-1}\|\right)\frac{2}{\sigma_{k}\gamma_k^2}}-\|w_{k-1}\|\right),
\end{align} 
then $\|\Delta\cR_{\jmax}\|\leq\varepsilon$.
\end{enumerate}
\end{subequations}
\end{theorem}
\begin{proof}
Consider the following estimate 
\begin{align}\nonumber
	\|\Delta\cR_{j_{\max}}^{\text{ADI}}\|&\leq2\|\eta_{j_{\max}}^{
\text{ADI}}\|=2\|S_{j_{\max}}\Gamma_{j_{\max}}Z_{j_{\max}}^*M^*\|\\\label{ADIresgapbound1}
	&\leq2\|S_{{j_{\max}}-1}\Gamma_{{j_{\max}}-1}Z_{{j_{\max}}-1}^*M^*\|+2\gamma_{j_{\max}}^2\|s_{j_{\max}}\|\|Mv_{j_{\max}}\|\\\label{ADIresgapbound}
	&\leq
2\sum\limits_{k=1}^{j_{\max}}\gamma_k^2\|s_k\|\|Mv_k\|.%\leq2\|M\|\sum\limits_{k=1}^j\frac{\gamma_k^2(\|w_{k-1}\|\|s_k\|+\|s_k\|^2)}{\sigma_{\min}(A+\alpha_kM)
% },
\end{align}
Moreover, 
\begin{align}\label{adi_estim_LS}
 %\|Mv_k\|\leq\|M\|\|(A+\alpha_k M)^{-1}(w_{k-1}-s_k)\|\leq \|M\| \frac{(\|w_{k-1}\|+\|s_k\|)}{\sigma_{\min}(A+\alpha_kM)}.
\|Mv_k\|=\|M(A+\alpha_k M)^{-1}(w_{k-1}-s_k)\|\leq \sigma_k(\|w_{k-1}\|+\|s_k\|).
\end{align}
 If the linear residual norms $\|s_k\|$ are so that each addend in the sum~\eqref{ADIresgapbound} is smaller than
$\frac{\varepsilon}{{j_{\max}}}$ we achieve $\|\eta_{j_{\max}}^{\text{ADI}}\|\leq\varepsilon/2$.
With 
%$\phi_k:=\displaystyle\frac{2\gamma_k^2\|M\|}{\sigma_{\min}(A+\alpha_kM)}$, 
$\phi_k:=\displaystyle 2\gamma_k^2\sigma_{k}$, 
$\omega_{k-1}:=\|w_{k-1}\|=\sqrt{\|\cR_{k-1}\|}$
% we want to push each addend
% in~\eqref{ADIresgapbound} below $\displaystyle\frac{\varepsilon}{j_{\max}}$ resulting in  $\|\eta_j^{\text{ADI}}\|\leq\varepsilon/2$.
this means we require $\phi_k(\omega_{k-1} \|s_k\|+\|s_k\|^2)\leq\displaystyle\frac{\varepsilon}{j_{\max}}.$
Hence, the desired largest allowed value $\|s_k\|$ is given by the positive root of the inherent quadratic equation 
$\phi_k(\omega_k \varsigma+\varsigma^2)-\frac{\varepsilon}{j_{\max}}=0$ such that
\begin{align}
	\|s_k\|\leq \half\left(\sqrt{\omega_{k-1}^2+\frac{4\varepsilon}{\phi_kj_{\max}}}-\omega_{k-1}\right)
\end{align}
leading to the desired result (a).
% \begin{align*}
%  \phi_k(\omega_k\|s_k\|+\|s_k\|^2)\leq
% \phi_k\left(\sqrt{\omega_k^2+\frac{4\varepsilon}{\phi_kj_{\max}}}-\half\omega_k^2+\tfrac{1}{4}\left(2\omega_k^2+\frac{4\varepsilon}{\phi_kj_
% { \max } } -2\omega_k\sqrt{\omega_k^2+\frac{4\varepsilon}{\phi_kj_{\max}}}\right)\right)\\
% =\frac{\varepsilon}{j_{\max}}
% \end{align*}
% and, consequently, $\|\eta^{\text{ADI}}_{j_{\max}}\|\leq\varepsilon$.
The second strategy can be similarly shown by using \eqref{ADIresgapbound1} and
\begin{align*}
	\|\eta^{\text{ADI}}_{k}\|\leq 2\|\eta^{\text{ADI}}_{k-1}\|+\phi_k(\|w_{k-1}\|\|s_k\|+\|s_k\|^2)
\end{align*}
and finding $\|s_k\|$ such that the right hand side in the above inequality is pushed below $\displaystyle\frac{k\varepsilon}{j_{\max}}$.
\end{proof}
The motivation behind the second stopping strategy~\eqref{ADI_relax2_theo} is that we can take the previous $\|\eta^{\text{ADI}}_{k-1}\|$  into
account. 
This can be helpful if at steps $i\leq k-1$ the used iterative method produced smaller linear residual norms than demanded, such that the linear residual
norms $\|s_i\|$ are allowed to grow slightly larger at later iteration steps $i>k-1$.
%%%%%%%%%%%%%%%%%%%%%%%%%%%%%
\paragraph{Convergence analysis of inexact ADI with relaxed solve tolerances}
Using the relaxation strategies in Theorem~\ref{thm:ADI_theorelax} reveals, under some conditions, further insight into the behaviour of the computed LR-ADI residual norms. 
\begin{theorem}[Decay of computed LR-ADI residuals]\label{thm:ADI_resdecay}
Assume $\|\cC_j\|<1$, $j=1,\ldots,\jmax$ such that the residuals $\|\cR^{\text{exact}}_{j}\|$ of the exact LR-ADI iteration decrease. Let the assumptions of Theorem \ref{thm:ADI_theorelax} hold. If the relaxation~\eqref{ADI_relax1_theo} is used in the inexact LR-ADI iteration, then 
\begin{align*}
\|\cR^{\text{comp.}}_{\jmax}\|\leq\|\cR^{\text{exact}}_{\jmax}\|+\varepsilon.
\end{align*}   
\end{theorem} 
\begin{proof}
For simplicity, let $r=1$, $M = I_n$ and $\lbrace\alpha_j\rbrace_{j=1}^{\jmax}\subset\R_-$. 
The assumption $\|\cC_j\|<1$ gives $\xi_j:=\gamma_j^2\sigma_j=\|\cC_j-I\|<2$ for $1\leq j\leq \jmax$. Consider the identity $w_j=\cC_jw_{j-1}+(\cC_j-I)s_j$ established in the proof of Lemma~\ref{Lem:adi_compres}. Then 
\begin{align*}
\|\cR^{\text{comp.}}_{j}\|&=\|w_j\|^2=\|\cC_jw_{j-1}\|^2+\|(\cC_j-I)s_j\|^2+2w_{j-1}^*\cC_j^*(\cC_j-I)s_j\\
&\leq\|\cC_jw_{j-1}\|^2+\xi_j^2\|s_j\|^2+2\omega_{j-1}\xi_j\|s_j\|,
\end{align*}
where we have used $\omega_{j-1} = \|w_{j-1}\|$. Inserting~\eqref{ADI_relax1_theo} yields
\begin{align*}
\xi_j^2\|s_j\|^2+2\omega_{j-1}\xi_j \|s_j\|&\leq \frac{\xi_j^2}{4}\left(\sqrt{\omega_{j-1}^2+\tfrac{2\varepsilon}{\xi_j j_{\max}}}-\omega_{j-1}\right)^2+\omega_{j-1}\xi_j\left(\sqrt{\omega_{j-1}^2+\tfrac{2\varepsilon}{\xi_j j_{\max}}}-\omega_{j-1}\right)\\
&=\frac{\xi_j^2}{4} \left(\omega_{j-1}^2+\tfrac{2\varepsilon}{\xi_j j_{\max}}-2\omega_{j-1}\sqrt{\omega_{j-1}^2+\tfrac{2\varepsilon}{\xi_j j_{\max}}}+\omega_{j-1}^2\right)\\&\quad\quad\quad\quad\quad\quad+\omega_{j-1}\xi_j\sqrt{\omega_{j-1}^2+\tfrac{2\varepsilon}{\xi_j j_{\max}}}-\omega_{j-1}^2\xi_j.
\end{align*}
Now, using $\xi_j<2$ we obtain
\begin{align*}
\xi_j^2\|s_j\|^2+2\omega_{j-1}\xi_j \|s_j\|
&\leq \frac{\xi_j}{2} \left(2\omega_{j-1}^2+\tfrac{2\varepsilon}{\xi_j j_{\max}}-2\omega_{j-1}\sqrt{\omega_{j-1}^2+\tfrac{2\varepsilon}{\xi_j j_{\max}}}\right)+\omega_{j-1}\xi_j\sqrt{\omega_{j-1}^2+\tfrac{2\varepsilon}{\xi_j j_{\max}}}-\omega_{j-1}^2\xi_j\\
&=\tfrac{\varepsilon}{j_{\max}}.
\end{align*}
resulting in $\|\cR^{\text{comp.}}_{j}\|\leq \|\cC_jw_{j-1}\|^2+\frac{\varepsilon}{j_{\max}}$.
%\begin{align*}
%\|\cR^{\text{comp.}}_{j}\|\leq \|\cC_jw_{j-1}\|^2+\frac{\varepsilon}{j_{\max}}\leq \|\cC_j\cC_{j-1}w_{j-2}\|^2+\sigma_{j-1}^2\|\cC_js_{j-1}\|+2\omega_{j-2}\sigma_{j-1}\|s_{j-1}\|+\frac{\varepsilon}{j_{\max}}
%\end{align*} 
Using Lemma~\ref{Lem:adi_compres} again for $w_{j-1}$ and repeating the above steps for $1\leq j\leq \jmax$ gives
\begin{align*}
\|\cR^{\text{comp.}}_{\jmax}\|&=\|w_{\jmax}\|^2=\|\prod\limits_{j=1}^{\jmax}\cC_jw_{0}\|^2+\sum\limits_{j=1}^{\jmax}\frac{\varepsilon}{j_{\max}}=\|\cR^{\text{exact}}_{\jmax}\|+\varepsilon.
\end{align*}
\end{proof}
Using Theorem~\ref{thm:ADI_resdecay} together with~\eqref{adi:resgap} yields the following conclusion.
\begin{corollary}
Under the same conditions as Theorem~\ref{thm:ADI_resdecay} we have $\|\cR^{\text{true}}_{\jmax}\|\leq \|\cR^{\text{exact}}_{\jmax}\|+2\varepsilon$.
\end{corollary}
Hence, if~\eqref{ADI_relax1_theo} is used, then the (true) Lyapunov residual norms in inexact LR-ADI are a small perturbation of the
residuals of the exact method.  
\begin{remark}
Another way to enforce decreasing computed Lyapunov residuals norms $\|\cR^{\text{comp.}}_{j}\|$
is using the stronger condition $\|\cC_j\|<1$ together with a proportional inner accuracy $\|s_j\|<\mu\|\cR^{\text{comp.}}_{j-1}\|$, $0<\mu<1$, similar as in~\cite{ShaSS15} in the context of  inexact stationary iterations. For the LR-ADI iteration this does, however, not ensure small residual gaps, $\|\Delta\cR^{\text{comp.}}_{j}\|<\varepsilon$, and, thus, no small true residuals.
\end{remark}

%%%%%%%%%%%%%%%%%%%%%%%%%%%%%%%%%%%%%%%%%%%%%%%%%%%%%%%%%%%%%%%%%%%%%%%%%%%%%%%%%%
\paragraph{Practical relaxation strategies for inexact LR-ADI}
The proposed stopping criteria~\eqref{ADI_relax_theo} are not very practical, since
the employed bound~\eqref{adi_estim_LS} will often give substantial overestimation of $\|Mv_k\|$ by several orders of magnitude which, in turn, will result in
smaller inner tolerances $\tau^{\text{LS}}$ than actually needed.  Furthermore, computing or estimating the %smallest singular value 
norm of the large matrix
$M(A+\alpha_kM)^{-1}$ is possible, e.g. by inexact Lanczos-type approaches, but the extra effort for this does not pay
off.
Here we proposed some variations of the above approaches that are better applicable in an actual implementation. From the algorithmic description
(Algorithm~\ref{alg:lradi}) of the LR-ADI
iteration it holds
\begin{align}
 \|Mv_k\|=\frac{1}{\gamma^2_k}\|w_k-w_{k-1}\|\leq\frac{1}{\gamma^2_k}(\|w_k\|+\|w_{k-1}\|).%(\cC_k-I)(w_{k-1)-s_k\|
\end{align}
In practice, the sequence $\lbrace \|w_k\|\rbrace=\lbrace \sqrt{\|\cR^{\text{comp.}}_{k}\|}\rbrace$ is often monotonically decreasing. Assuming
$\|w_k\|\leq\|w_{k-1}\|$ suggests to use
$\|Mv_k\|\leq \frac{2\|w_{k-1}\|}{\gamma_k^2}$ in~\eqref{ADIresgapbound} leading to the relaxation criterion
\begin{subequations}\label{relaxadi_prac}
\begin{align}\label{relaxadi_prac1}
	\|s_k\|\leq\tau^{\text{LS}}_k=\frac{\varepsilon}{4j_{\max}\sqrt{\|\cR^{\text{comp.}}_{k-1}\|}}.
\end{align}
Starting from~\eqref{ADIresgapbound1},
assuming $\|\eta_{k-1}^{\text{ADI}}\|\leq \frac{(k-1)\varepsilon}{2\jmax}$, and enforcing
$\|\Delta\cR_k^{\text{ADI}}\|<\frac{k\varepsilon}{\jmax}$ we obtain
\begin{align}\label{relaxadi_prac2}
	\|s_k\|\leq\tau^{\text{LS}}_k=\frac{\frac{k\varepsilon}{\jmax}-2\|\eta_{k-1}^{\text{ADI}}\|}{4\sqrt{\|\cR^{\text{comp.}}_{k-1}\|}}.
\end{align}
\end{subequations}
The second relaxation strategy~\eqref{relaxadi_prac2} requires $\|\eta_{k-1}^{\text{ADI}}\|=\|S_{k-1}\Gamma_{k-1}Z_{k-1}^*M^*\|$ or an approximation thereof. From $\|\eta_{k-1}^{\text{ADI}}\|\leq \|\eta_{k-2}^{\text{ADI}}\|+\gamma_{k-1}^2\|Mv_{k-1}\|\|s_{k-1}\|$ 
basic bounds $\|\eta_{i}^{\text{ADI}}\|\leq u_i$, $i=0,\ldots,k-1$ can be computed in each step via
\begin{align}\label{relaxadi_update}
\|\eta_{k-1}^{\text{ADI}}\|\leq u_{k-1}:=u_{k-2}+\gamma_{k-1}^2\|Mv_{k-1}\|\|s_{k-1}\|,\quad u_0:=0.
\end{align}

Note that $Mv_{k-1}$ is required anyway to continue Algorithm~\ref{alg:lradi}.  
The linear residual  norms $\|s_i\|,~1\leq
i\leq k-1$ are sometimes available as byproducts of Krylov subspace solvers for linear systems if right preconditioning is used. In case of other forms of
preconditioning, the $s_i$ might need to be computed explicitly, which requires extra matrix vector products with $A$ (and $M$), or the  norms $\|s_i\|$ have to
be estimated in some other way. Similar to RKSM, for problems defined by real but unsymmetric coefficients, pairs of complex conjugated shifts can occur in
LR-ADI. These can be dealt with efficiently using the machinery developed in~\cite{BenKS13,morBenKS13,Kue16} to ensure that the majority of operations remains
in
real arithmetic. By following these results it is easily shown that if steps $k-2, k-1$ used a complex conjugated pairs of shifts, then in the
formula~\eqref{relaxadi_update} the real and imaginary parts of both $v_{k-2},s_{k-2}$ enter the update.

At the first look, \eqref{relaxadi_prac} appears to allow somewhat less relaxation compared to RKSM since only the square roots of the computed Lyapunov
residual norms appear in the denominator. However, the numerical examples in the next section show that with these relaxation strategies we can reduce the
amount of work for inexact LR-ADI by up to 50 per cent.

\section{Numerical examples}
\label{sec:num}
In this section we consider several numerical examples and apply inexact RKSM and inexact LR-ADI with our practical relaxation strategies.
The experiments were carried out in \matlab~2016a on a \intel\coretwo~i7-7500U CPU @ 2.7GHz with 16 GB RAM.
We wish to obtain an approximate solution such that the scaled true Lyapunov residual norm satisfies
\begin{align*}
 \fR:=\|\cR^{\text{true}}\|/\|B\|^2\leq \hat\varepsilon,\quad 0<\hat\varepsilon\ll 1,
\end{align*}
i.e., $\varepsilon=\hat\varepsilon\|B\|^2$. In all tests we desire to achieve this accuracy with $\hat\varepsilon=10^{-8}$ within at most $j_{\max}=50$ iteration
steps. In most cases we employ dynamic shift generation strategies using the approach in~\cite{DruS11} for RKSM and shifts based on a projection
idea for LR-ADI, see ~\cite{BenKS14,Kue16} for details. For the latter one, the projection basis is chosen as the last $\min(\mathrm{coldim}(Z_r),10r)$ columns
of the low-rank factor $Z_j$. 
The exception is one upcoming symmetric problem where optimal shifts following~\cite{Wac13} for the LR-ADI iteration are used. For comparative purposes, the exact EKSM is also tested for some examples.

We employ different Krylov subspace solvers and stopping criteria for the arising linear systems of equations, in
particular, we use fixed as well as dynamically chosen inner solver tolerances $\tau^{\text{LS}}$ with the proposed relaxation strategies. 
The fixed solve tolerances were determined a-priori via trial-and-error such that the residual behavior of inexact method mimicked the residuals of the method using sparse-direct solution approaches for the linear systems. 
If not stated otherwise, the values $\tau^{\text{LS}}_{\min}=10^{-12}$, $\tau^{\text{LS}}_{\max}=0.1$ are taken as minimal and maximal linear solve tolerances.
As Krylov subspace solvers we use BiCGstab and MINRES for problems with unsymmetric and symmetric ($A=A^*$, $M=M^*$) coefficients, respectively. Sparse-direct
solves were carried out by the \matlab~\textit{backslash}-routine, or, in case of EKSM, by precomputed sparse LU (or Cholesky) factorizations of $A$. EKSM and RKSM handled the generalized problems by means of the equivalent problem using sparse Cholesky factors of the matrix $M$ as described in Section~\ref{ssec:imple_gen}.
\begin{table}[t]
  \footnotesize
  \centering
  \caption{Properties and setting of used test equations including save guard constant $\delta$ in~\eqref{relax_sk_prac} and type of
preconditioner used. Here, iLU$(X,\nu)$ and iC$(X,\nu)$ refer to incomplete LU and, respectively, Cholesky factorization of the matrix  $X$ with drop
tolerance $\nu$.}
\setlength{\tabcolsep}{0.5em}
\renewcommand{\arraystretch}{1.1}
  \begin{tabularx}{\textwidth}{|l|l|l|l|X|l|l|}
    \hline
    Example & $n$ & $r$&matrices&description&$\delta$&prec.\\
    \hline%\hline
    \texttt{cd2d}&40000&1&$A\neq A^*$, $M=I$& finite difference discretization of $\Delta u - 100 x \frac{\partial{u}}{\partial x} - 200 y
\frac{\partial{u}}{\partial
y}$ on $[0,1]^2$&0.01&iLU(A,10$^{-3}$)\\
%\texttt{sprand}&10000&4&$A\neq A^*$, $M=I$& random sparse matrix from~\cite{}&0.5&--\\
    \texttt{heat3d}&125000&4&$A=A^*$, $M=I$& finite difference discretization of heat equation on $[0,1]^3$&1&iC(-A,10$^{-2}$)\\
    \texttt{fem3d}&24389&1&$A\neq A^*$, $M=M^*\neq I$&finite element model of 3d convection-diffusion problem from~\cite{BenHSetal16}&$\half$&iLU(A,10$^{-2}$)\\
		\texttt{msd}&24002&4&$A\neq A^*$, $M=M^*\neq I$&modified version of second order dynamical system from~\cite{TruV09}&$0.1$&iLU($A_j$,10$^{-2}$)\\
    \texttt{fem3d-are}&24389&1&$A\neq A^*$, $M=M^*\neq I$&extension of \texttt{fem3d} to ARE~\eqref{care} from~\cite{BenHSetal16} including
$C\in\Rn$&$\half$&iLU(A,10$^{-2}$)\\
    \hline
  \end{tabularx}\label{tab:examples}
\end{table}
We consider five examples, two standard Lyapunov equations (\texttt{cd2d} and \texttt{heat3d}, where \texttt{heat3d} is an example for which $r>1$), two
generalized Lyapunov equation (\texttt{fem3d}, \texttt{msd}) and an algebraic Riccati equation (\texttt{fem3d-are}), in order to illustrate the theoretical results in this
paper. Details and setup on the examples we used are summarized in Table~\ref{tab:examples}. The matrices $B$ for examples \texttt{cd2d} and \texttt{heat3d} 
were generated randomly with a standard Gaussian distribution and the initialization \texttt{randn('state', 0)}. 
Examples \texttt{fem3d}, \texttt{fem3d-are}
provide vectors $B,C$~\cite{BenHSetal16}.
Because of the symmetric coefficient in \texttt{heat3d}, we use this example to experiment with precomputed shifts from~\cite{Wac13} instead of the dynamic shifts in all other experiments.
The example \texttt{msd} represents coefficients of a first order linearization of a second order dynamical system and is a modified version of an example used in~\cite{TruV09}. Further details on the setup of this example can be found in Appendix~\ref{sec:appB}. For efficient iterative solves, this example required updating the preconditioner in every iteration step.

\begin{figure}[ht!]
\caption{Illustration of the decay of the Galerkin solution for Example \texttt{cd2d}. Left: Row norms of $Y_j$, $H_j^{-1}Y_j$ and the corresponding bounds
(Corollary~\ref{corYentrybound}, Corollary~\ref{cor:deltahybound}) against the iteration number $j$. Right: Absolute values of the entries of the final $Y_j$.}
\label{fig:decayY}
\includegraphics[width=0.5\linewidth]{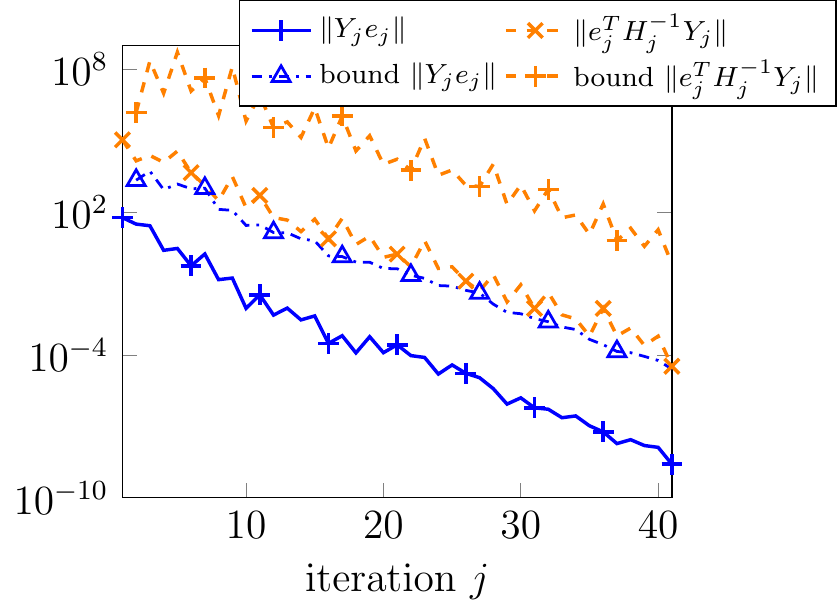}~\includegraphics[width=0.5\linewidth]{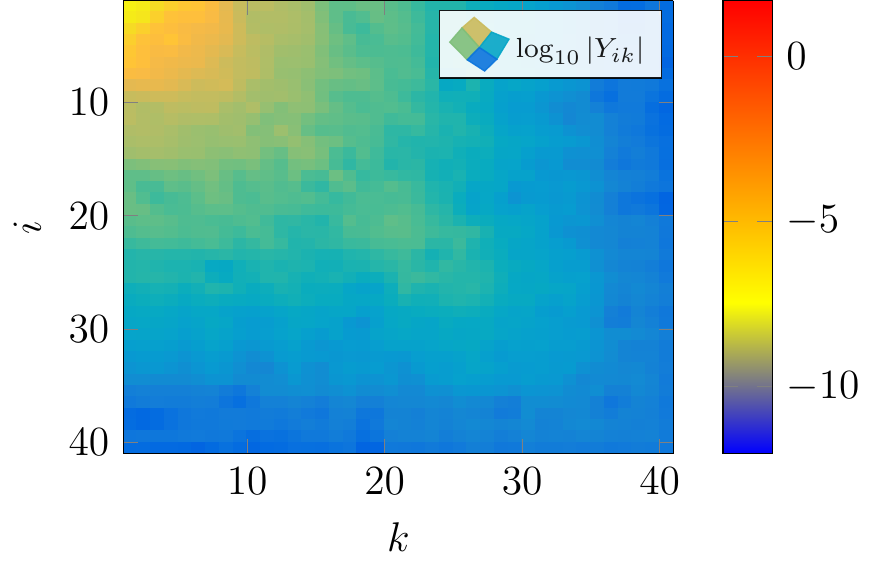}
\end{figure}

At first we briefly investigate the results in Section~\ref{sec:props} on the decay of Galerkin solution $Y_j$ using example \texttt{cd2d}.
We run the exact RKSM and plot the row norms of $Y_j$, $H_j^{-1}Y_j$ and the corresponding bounds obtained in Corollaries~\ref{corYentrybound}
and~\ref{cor:deltahybound} against the iteration number $j$, as well as the absolute values of the entries of the final Galerkin solution $Y_j$ in the left plot
of Figure~\ref{fig:decayY}. Real shift parameters were used for this experiment. The figures are an example to show that our theoretical bounds can indeed be
verified, but that they significantly overestimate the true norms. The right plot shows the decay of the entries $Y_j$ as predicted
by~Corollary~\ref{corYentrybound}. Similar results are obtained for the other examples.

\begin{table}[h!]
  \small
  \centering
  \caption{Experimental results. The columns denote the outer (i.e., EKSM, RKSM or LR-ADI) and inner method (which Krylov subspace method), the used inner stopping criterion (fixed or relaxed), the smallest and largest generated inner solve tolerances $\min\tau^{\text{LS}}$, $\max\tau^{\text{LS}}$, the
number of required outer iterations, column dimension of the low-rank solution factor, the final obtain scaled computed residual norm $\fR^{\text{comp}}_{j}$,
 the difference $\delta\fR_j:=\vert\fR^{\text{comp}}_j-\fR_j^{\text{true}}\vert$ between the true and the computed Lyapunov residual norm, the total number of inner iteration steps, 
the achieved relative savings regarding the amount of inner iteration steps compared to the run with a fixed inner tolerance, and the computing times in seconds. 
}\label{tab:ex_results}
  \setlength{\tabcolsep}{0.1em}
  \begin{tabularx}{\textwidth}{|X|l|l|r|r|r|r|r|r|r|r|r|r|}
    \hline   
Ex.&outer&inner&stop&$\min\tau^{\text{LS}}$&$\max\tau^{\text{LS}}$&it\textsuperscript{out}&dim&$\fR^{\text{comp}}_{j}$&$\delta\fR_j$&it\textsuperscript{in}
&save&time\\
    \hline\hline
    \multirow{9}{*}{
    \rotatebox{90}{
  \texttt{cd2d}
    }}    
&RKSM&direct&--&--&--&37&37&2.7e-09&&--&--&5.5\\ 
&RKSM&BICGSTAB&fixed&\multicolumn{2}{c|}{1.00e-10}&37&37&3.2e-09&2.0e-10&832&&5.3\\ 
&RKSM&BICGSTAB&relax \eqref{relax_sk_prac1}&2.0e-12&2.8e-02&37&37&4.2e-09&2.2e-09&618&25.7\%&3.9\\ 
&RKSM&BICGSTAB&relax \eqref{relax_sk_prac2}&2.0e-12&9.1e-03&37&37&3.0e-09&7.6e-09&614&26.2\%&3.7\\ 
\cline{2-13}
&EKSM&direct&--&--&--&40&80&3.6e-09&&--&--&1.5\\ 
\cline{2-13}
&LR-ADI&direct&--&--&--&49&49&7.0e-09&&--&--&6.1\\ 
&LR-ADI&BICGSTAB&fixed&\multicolumn{2}{c|}{1.0e-10}&49&49&7.0e-09&1.4e-12&1382&&6.8\\ 
&LR-ADI&BICGSTAB&relax \eqref{relaxadi_prac1}&5.0e-11&4.9e-03&49&49&7.0e-09&2.4e-11&1004&27.4\%&4.7\\ 
&LR-ADI&BICGSTAB&relax \eqref{relaxadi_prac2}&5.0e-11&2.1e-02&49&49&6.9e-09&1.7e-10&971&29.7\%&4.5\\ 
    \hline
    \multirow{9}{*}{%
%     \begin{minipage}{0.25\linewidth}
\rotatebox{90}{
      \texttt{heat3d}}
%     \end{minipage}
    }    
&RKSM&direct&--&--&--&15&60&7.9e-09&&--&--&69.0\\ 
&RKSM&MINRES&fixed&\multicolumn{2}{c|}{1.00e-09}&15&60&7.9e-09&5.3e-14&492&&20.1\\ 
&RKSM&MINRES&relax \eqref{relax_sk_prac1}&2.0e-10&1.0e-01&15&60&7.9e-09&2.6e-10&265&46.1\%&13.2\\ 
&RKSM&MINRES&relax \eqref{relax_sk_prac2}&2.0e-10&1.0e-01&15&60&7.9e-09&2.9e-10&252&48.8\%&12.2\\ 
\cline{2-13}
&EKSM&direct&--&--&--&20&160&3.9e-11&&--&--&23.9\\ 
\cline{2-13}
&LR-ADI&direct&--&--&--&16&64&4.4e-09&&--&--&77.4\\ 
&LR-ADI&MINRES&fixed&\multicolumn{2}{c|}{1.0e-09}&16&64&4.4e-09&9.0e-14&439&&17.2\\ 
&LR-ADI&MINRES&relax \eqref{relaxadi_prac1}&5.0e-11&1.1e-03&16&64&4.4e-09&8.7e-14&301&31.4\%&13.5\\ 
&LR-ADI&MINRES&relax \eqref{relaxadi_prac2}&5.0e-11&7.7e-03&16&64&4.4e-09&1.3e-11&270&38.5\%&12.2\\ 
  \hline
\multirow{9}{*}{%
    %\begin{minipage}{0.25\linewidth}
  \rotatebox{90}{\texttt{fem3d}}
    %\end{minipage}
    }   	
&RKSM&direct&--&--&--&23&23&3.5e-09&&--&--&44.8\\ 
&RKSM&BICGSTAB&fixed&\multicolumn{2}{c|}{2.00e-10}&23&23&4.0e-09&1.1e-10&474&&12.4\\ 
&RKSM&BICGSTAB&relax \eqref{relax_sk_prac1}&3.7e-11&4.5e-02&23&23&6.2e-09&1.2e-09&309&34.8\%&9.8\\ 
&RKSM&BICGSTAB&relax \eqref{relax_sk_prac2}&3.9e-11&1.0e-01&23&23&4.7e-09&1.4e-10&299&36.9\%&9.9\\ 
\cline{2-13}
&EKSM&direct&--&--&--&25&50&1.9e-09&&--&--&5.5\\ 
\cline{2-13}
&LR-ADI&direct&--&--&--&27&27&4.2e-09&&--&--&52.2\\ 
&LR-ADI&BICGSTAB&fixed&\multicolumn{2}{c|}{2.0e-10}&27&27&4.2e-09&1.8e-14&359&&2.2\\ 
&LR-ADI&BICGSTAB&relax \eqref{relaxadi_prac1}&5.0e-11&2.7e-04&27&27&4.2e-09&1.8e-11&246&31.5\%&1.6\\ 
&LR-ADI&BICGSTAB&relax \eqref{relaxadi_prac2}&5.0e-11&2.8e-03&27&27&4.4e-09&2.5e-10&209&41.8\%&1.5\\ 
\hline
\multirow{9}{*}{%
    %\begin{minipage}{0.25\linewidth}
  \rotatebox{90}{\texttt{msd}}
    %\end{minipage}
    }   	
&RKSM&direct&--&--&--&86&344&9.5e-09&&--&--&120.9\\ 
&RKSM&BICGSTAB&fixed&\multicolumn{2}{c|}{6.67e-11}&86&344&9.6e-09&2.9e-08&738&&53.6\\ 
&RKSM&BICGSTAB&relax \eqref{relax_sk_prac1}&2.7e-12&1.4e-02&86&344&8.7e-09&5.8e-12&668&9.5\%&51.0\\ 
&RKSM&BICGSTAB&relax \eqref{relax_sk_prac2}&1.0e-12&3.3e-03&88&352&7.9e-09&2.3e-12&696&5.7\%&53.0\\ 
\cline{2-13}
&EKSM&direct&--&--&--&70&560&1.5e-09&&--&--&17.5\\ 
\cline{2-13}
&LR-ADI&direct&--&--&--&66&264&1.0e-08&&--&--&164.4\\ 
&LR-ADI&BICGSTAB&fixed&\multicolumn{2}{c|}{6.7e-11}&59&236&3.0e-09&1.9e-11&1683&&69.3\\ 
&LR-ADI&BICGSTAB&relax \eqref{relaxadi_prac1}&2.5e-11&1.5e-03&62&248&9.9e-09&1.1e-13&1451&13.8\%&55.4\\ 
&LR-ADI&BICGSTAB&relax \eqref{relaxadi_prac2}&2.5e-11&1.5e-02&65&260&4.7e-10&1.7e-10&1065&36.7\%&43.0\\
\hline

\multirow{3}{*}{%\rotatebox{90}
    %\begin{minipage}{0.25\linewidth}
    \rotatebox{90}{\texttt{fem3d-are}}
    %\end{minipage}
    }  
&RKSM&direct&--&--&--&22&22&5.8e-09&&--&--&43.6\\ 
&RKSM&BICGSTAB&fixed&\multicolumn{2}{c|}{2.00e-10}&22&22&5.8e-09&4.8e-14&436&&34.8\\ 
&RKSM&BICGSTAB&relax (\ref{relax_sk_prac1})&1.1e-10&1.0e-01&22&22&6.7e-09&2.3e-09&222&49.1\%&12.9\\ 
&RKSM&BICGSTAB&relax (\ref{relax_sk_prac2})&1.1e-10&1.0e-01&22&22&6.7e-09&2.3e-09&222&49.1\%&12.9\\ 
\hline
\end{tabularx}
\end{table}

We now experiment with the different practical relaxation strategies (\ref{relax_sk_prac}) and (\ref{relaxadi_prac}) from Sections~\ref{sec:rksm} and~\ref{sec:adi}
for the inner iteration. In Table~\ref{tab:ex_results} we report the results for all examples. 
There, we give, among other relevant information on the performance of the outer method under inexact inner solves, also the
final obtained scaled computed residual norms $\fR^{\text{comp}}_{j}$ (using the formula~\eqref{exactRKres} for RKSM and~\eqref{adi_resexact} for LR-ADI). For
assessing the reliability of the value of $\fR^{\text{comp}}_{j}$, the distance
$\delta\fR_j:=\vert\fR^{\text{comp}}_j-\fR_j^{\text{true}}\vert$ to the true scaled residual norms is also listed, where $\fR_j^{\text{true}}$ was computed
using the Lanczos process on $\cR_j$.

First, we observe that in all examples the difference between the true and the computed Lyapunov residual norm, $\delta\fR_j$, is of the order $\cO(10^{-9})$, or smaller. 

The second observation we make is that both practical relaxation criteria for RKSM, namely (\ref{relax_sk_prac1}) and (\ref{relax_sk_prac2}) are effective and
lead to a reduction in overall inner iteration numbers. They both lead to nearly the same results for the number of inner iterations, the gain in using
(\ref{relax_sk_prac2}) over (\ref{relax_sk_prac1}) is very minor. For our examples, by using the relaxed stopping criterion up to $50$ per cent of inner iterations can by saved compared to a fixed inner tolerance.

For the LR-ADI method we consider the two relaxation criteria (\ref{relaxadi_prac1}) and (\ref{relaxadi_prac2}). Again, we observe a reduction in the total number of iterations for both relaxation strategies, but we also see that the second relaxation criterion (\ref{relaxadi_prac2}) reduces the iteration numbers even
further compared to  (\ref{relaxadi_prac1}), so the use of (\ref{relaxadi_prac2}) is generally recommended. Here it pays off that the second strategy takes the previous iterate into account. The total savings in inner iterations between fixed
and relaxed tolerance solves for LR-ADI is between $14$ and $42$ per cent in our examples. 
% We also point out that the nearly indistinguishable residual norms

The fairer comparison between the fixed and relaxed tolerance solve is the computation time; in Table~\ref{tab:ex_results} we report the computation time for
direct solves as well as iterative solves with fixed and relaxed solve tolerances. For all examples we see that the relaxed solve tolerance also
leads to (sometimes significant) savings in computation time. 
For Example \texttt{fem3d}, the generalized Lyapunov equation, the absolute time saving is not so significant
(both for RKSM and LR-ADI), however, for the other three examples the savings are quite large, in particular for inexact RKSM. The results for the Riccati
example \texttt{fem3d-are} indicate that the proposed relaxation criteria also work for inexact RKSM for Riccati equations.
Note that for Example
\texttt{cd2d}, the direct solver outperforms the iterative methods with fixed small solve tolerance (both for RKSM and LR-ADI), however, the relaxed tolerance
versions perform similar to the direct methods in terms of computation time. This is to be expected as example
\texttt{cd2d} represents a  two-dimensional problem, where sparse-direct solvers are usually very efficient. For the three-dimensional examples
\texttt{heat3d}, \texttt{fem3d} and \texttt{fem3d-are} the iterative solvers significantly outperform the direct method.
In some cases we observe that the inexact methods require slightly more outer iteration steps (e.g., RKSM with relaxation~\eqref{relax_sk_prac2} for example~\texttt{msd}), indicating that further fine tuning of the relaxation parameters might be necessary. Sometimes  the number of outer steps is lower for the inexact method (inexact LR-ADI for example~\texttt{msd}).
Regarding the comparison with EKSM, we observe that EKSM achieves sometimes the smallest computations times (e.g., examples~\texttt{cd2d}, \texttt{msd}) but always generates the largest subspace dimensions among all tested methods.
Using a block-MINRES for the \texttt{heat3d} example (with $r=4$) largely led to very similar results in terms of the numbers of inner iteration steps, but the
overall computing times were slightly larger. %For the \texttt{heat3d}

As interesting side observation, we point out that RKSM, LR-ADI generated new shift parameters on their own in each run, which still did not lead to substantial differences
regarding the Lyapunov residual behavior, although the adaptive shift generation techniques~\cite{DruS11,BenKS14} are based on the eigenvalues of $T_j$
(rational Ritz values of $A$) and, hence, depend on the built up subspace. One explanation might be that the eigenvalue generation itself can be seen as inexact
rational Arnoldi process for the eigenvalue problem and if the inner solve tolerances are chosen intelligently, no large differences regarding the Ritz
values should appear~\cite{LehM98,Sim05, FrSp09}.
For the \texttt{heat3d} example precomputed shifts~\cite{Wac13} were used within LR-ADI leading to almost indistinguishable residual curves (cf. right plot in Figure~\ref{fig:heat3d_res}) indicating that the optimal shift approaches for certain problem classes still work under inexact solves.
Similar observations were made for the remaining examples when precomputed shifts were used for both exact and inexact methods.

\begin{figure}
\caption{Results for Example~\texttt{heat3d}: Scaled computed residual norms $\fR^{\text{comp}}$ and inner tolerances $\tau^{\text{LS}}$ vs iteration numbers
obtained by (in)exact RKSM (left) and LR-ADI (right). Here, LR-ADI uses precomputed shifts.}
\label{fig:heat3d_res}
\includegraphics[width=0.95\textwidth]{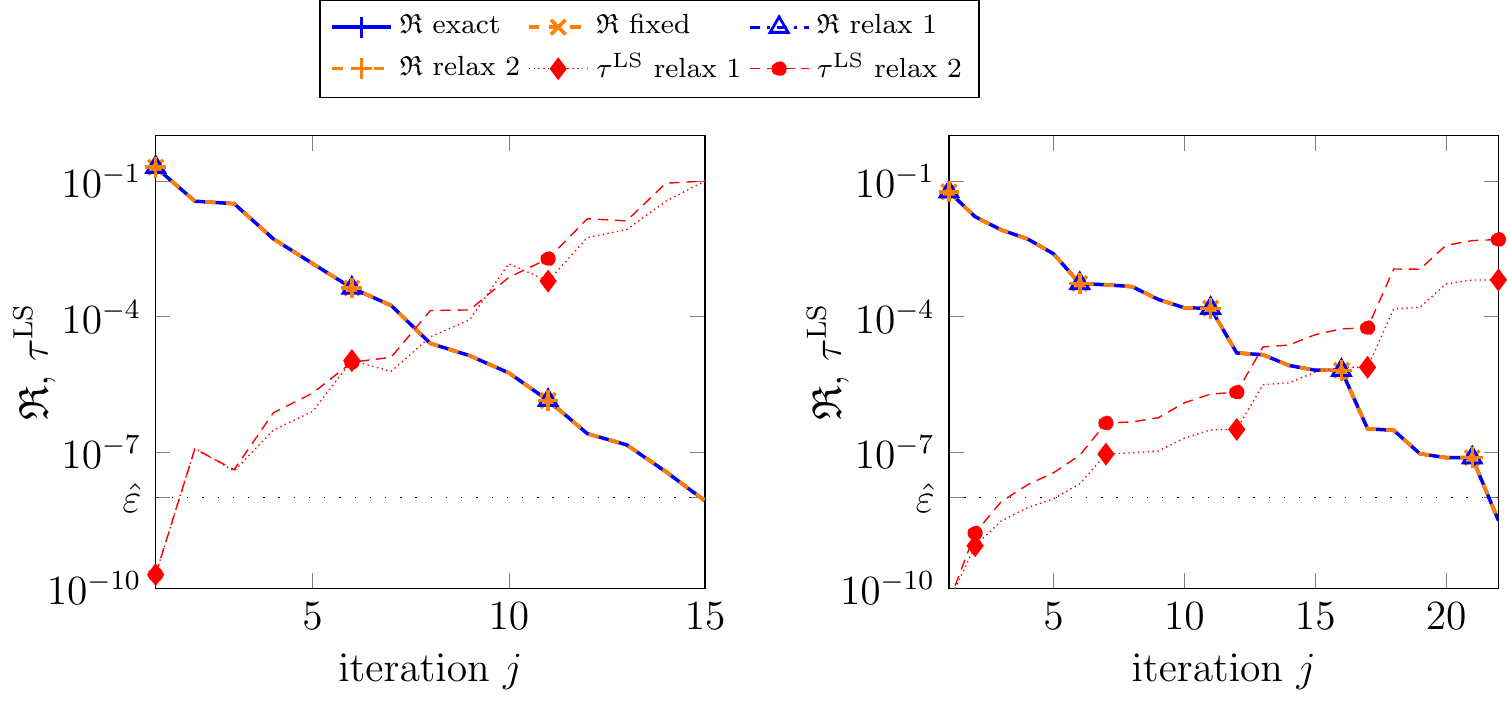}%~\includegraphics[width=0.5\linewidth]{images/heat3d_resADI.pdf}
\end{figure}

Figure~\ref{fig:heat3d_res} shows the scaled computed residual norms $\fR^{\text{comp}}$ and inner tolerances $\tau^{\text{LS}}$ vs iteration numbers obtained
by (in)exact RKSM and LR-ADI for Example~\texttt{heat3d}. 

On the left we plot the convergence history for the exact RKSM (e.g. direct solves within the inner iteration), the inexact RKSM with fixed small solve
tolerance within the iterative solve and the inexact RKSM with relaxation criterion (\ref{relax_sk_prac1}) and (\ref{relax_sk_prac2}) within the iterative
solution of the inner linear system. All computed residual norms are decreasing and virtually indistinguishable. The solve tolerances when using relaxation
criterion (\ref{relax_sk_prac1}) are shown in dotted lines with diamonds and the ones using criterion (\ref{relax_sk_prac2}) are shown in dashed lines with red
circles. The relaxation criteria lead to increasing inner solve tolerances, but as already observed in Table~\ref{tab:ex_results}, both criteria
for inexact RKSM give nearly the same results.

The right plot in Figure~\ref{fig:heat3d_res} shows the same results for LR-ADI. Again,the convergence history of the residual norms for inexact LR-ADI using the two relaxation strategies is not visibly distinguishable from the exact LR-ADI. However, we observe that the second relaxation criterion (\ref{relaxadi_prac2}), shown in dashed lines with red circles, gives better results, e.g. more relaxation and hence fewer inner iterations, than the first relaxation criterion (\ref{relaxadi_prac1}), shown in dotted lines with
diamonds, a result also observed in Table~\ref{tab:ex_results}. Similar plots as in Figure~\ref{fig:heat3d_res} can be obtained for other examples.

%%%%%%%%%%%%%%%%%%%%%%%%%%%%%%%%%%%%%%%%%%

\section{Conclusions and future work}\label{sec:concl}
The numerical solution of large scale Lyapunov equations is a very important problem in many applications. The rational Krylov subspace method (RKSM) and the low-rank alternating directions implicit (LR-ADI) iteration are well-established methods for computing
low-rank solution factors of large-scale Lyapunov equations. The main task in both those methods is to solve a linear system at each step, which is usually carried out iteratively and hence inexactly. 

We observed empirically that, when solving the linear system at each iteration step, the accuracy of the solve may be relaxed while maintaining
the convergence to the solution of the Lyapunov equation. In this paper we have presented a theoretical foundation for explaining this phenomenon, both for the
inexact RKSM method and the inexact low-rank ADI iteration. For both methods we introduced a so-called residual gap, which depends on the accuracy of the linear
system solve and on quantities arising from the solution methods for the large scale Lyapunov equation. We analyzed this gap for each method which provided
theoretical relaxation criteria for both inexact RKSM and inexact ADI. These criteria are often not applicable in practice as they contain unknown and/or
overestimated quantities. Hence, we gave   
practical relaxation criteria for both methods. Our numerical results indicate that using flexible accuracies gives very good results and can reduce the
amount of work for solving large scale Lyapunov equations by up to 50 per cent. 

One numerical experiment with inexact RKSM indicated that relaxation strategies might also be fruitful for low-rank methods for algebraic Riccati
equations~\cite{BenS13,BenKS14,SimSM14,Sim16,BenHSetal16,BenBKetal18}, making this an obvious future research topic, together with inexact linear solves in
low-rank methods for
Sylvester equations~\cite{BenLT09,BenK14,Kue16}. In this work, we restricted ourselves to standard preconditioning techniques. Improved concepts such as \textit{tuned}
preconditioners and similar ideas~\cite{FrSp09,BerM17} might further enhance the performance of the inner iteration process.
Preliminary tests, however, did not yield any performance gain from these techniques worth mentioning, further investigations are necessary in this
direction.
A further research direction worth pursuing is to reduce the computational effort for solving the sequences of shifted linear systems by storing the Krylov
basis
obtained from solving one (e.g., the first) linear system and employing subspace recycling techniques as, e.g., discussed for LR-ADI in~\cite{morLi00} and
for rational Krylov methods in the context of model
order reduction in~\cite{AhmSG12}. 
Allowing inexact matrix vector products, and in case of generalized equations also inexact solves with $M$, represents a further, more challenging research
perspective.

\paragraph{Acknowledgements} The authors are grateful to Cost Action EU-MORNET (TD1307) and the Department of Mathematical Sciences at Bath, that provided
funding for research visits of PK to the University of Bath, where substantial parts of this work have been conducted. 
This work was primarily generated while PK was affiliated with the Max Planck Institute for Dynamics of Complex Technical Systems in Magdeburg, Germany. 
Furthermore, the authors thank Kirk
Soodhalter (Trinity College Dublin) for 
insightful discussion regarding block Krylov subspace methods. Moreover, the authors would like to thank the referees for their valuable and insightful comments which helped to improve the manuscript.

\appendix
\section{Proof of Lemma~\ref{Lem:ekinvHj}}
\label{sec:appA}
% The following lemma shows that the first $1< j\leq k$ entries of $e_j^*H_k^{-1}$ are essentially determined by the left null space of
% $H_k$ augmented by one row and $e_{j-1}^*H_{j-1}^{-1}$ modulo scaling.
% \begin{lemma}\label{Lem:ekinvHj}
%  Let  $H\in\C^{k+1,k}$ be an unreduced upper Hessenberg matrix with $H_j:=H(1:j,1:j)$ nonsingular $\forall 1\leq j\leq k$, and 
%  let $\omega\in\C^{1\times k+1}$ s.t. $\omega H=0$. Define the vectors
% $f^{(j)}_m:=e_j^*H_m^{-1}$, $1\leq j,m\leq k$. Then
% it holds for $1\leq j\leq k$
% \begin{subequations}
%  \begin{align}\label{ekinHj_fjj}
%  f^{(j)}_j&=-\frac{\omega(1:j)}{\omega(j+1)h_{j+1,j}},\\
% \begin{split}\label{ekinHj_fjk}
%   f^{(j)}_k&=\frac{v^{(j)}_k}{\phi_{k}^{(j)}},\quad \text{where} \\
%   v^{(j)}_k&:=\omega(1:k)+[0_{1,j},[0_{1,k-j-1},h_{k+1,k}\omega(k+1)]H(j+1:k,j+1:k)^{-1}],\quad\phi_{k}^{(j)}:=v^{(j)}_kH_{k}e_j.
% \end{split}
% \end{align}
% Moreover, for $j>1$, the first entries of $v^{(j)}_k$ can be expressed by
%  \begin{align}\label{ekinHj_fold}
%   v^{(j)}_k(1:j)&=[-h_{j,j-1}f^{(j-1)}_{j-1},1]\omega(j).
% \end{align}
% \end{subequations}
% \end{lemma}
\begin{proof}%[Proof of Lemma~\ref{Lem:ekinvHj}]
From $\omega \underline{H}_j=0$ we have for $k\leq j$
 \begin{align*}
 0=\omega\smb \begin{smallmatrix}
           H_{k}\\
           e_{k}^*h_{k+1,k}\\
           0_{j-k,k}
          \end{smallmatrix}\Bigg|
          H_{:,k+1:j}
          \sme,
%           \smb
%           H_{k}&\\
%           & \Theta
%           \sme^{-1},~\Theta\in\C^{j-k,j-k}~\text{nonsingular}.
          \end{align*}
which, for $k=j$, immediately leads to $f_j^{(j)}=-\omega_{1:j}/(h_{j+1,j}\omega_{j+1})$, the first equality in~\eqref{ekinHj_fjj}.  
Similarly, it is easy to show that for $k<j$, a left null space vector $\hat\omega\in\C^{1\times k+1}$ of $\underline{H}_{k}$ is given by the first $k+1$
entries of the null vector
$\omega$ of $\underline{H}_j$. Hence, $f^{(k)}_k=-\omega_{1:k}/(h_{k+1,k}\omega_{k+1})$ holds for all $k\leq j$. 

For computing $f_j^{(k)}= e_k^*H_j^{-1}$, $k<j$, we use the following partition of $H_j$ and consider the splitting $f^{(k)}_j = [u,y]$, $u\in\C^{1\times k}$, $y\in\C^{1\times k-j}$:
\begin{equation*}
 e_k^*=\smb 0_{1,k-1}&\big|1\big|&0_{1,j-k}\sme=f^{(k)}_jH_j=[u,y]\smb \begin{smallmatrix}
           H_{k-1}\\
           e_{k-1}^*h_{k,k-1}\\
           0_{j-k,k-1}
          \end{smallmatrix}\Bigg|
          H_{:,k:j}
          \sme=\smb u\smb H_{k-1}\\
           e_{k-1}^*h_{k,k-1}
           \sme&\Bigg|[u,y]H_{:,k:j}\sme.
\end{equation*}
This structure enforces conditions on $[u,y]$, which we now explore. First, $u$ has to be a multiple of $\omega_{1:k}$. Here we exploited that due to the Hessenberg structure,
$\omega_{1:k}\underline{H}_{k-1}=0$. In particular,
\begin{align*}
 u_{1:k-1}H_{k-1}+u_ke_{k-1}^*h_{k,k-1}=0,
\end{align*}
such that $u_kh_{k,k-1}e_{k-1}^*H_{k-1}^{-1}=-u_{1:k-1}$. Since $f^{(k-1)}_{k-1}=e_{k-1}^*H_{k-1}^{-1}$
we can infer $u_{1:k-1}=-u_kh_{k,k-1}f^{(k-1)}_{k-1}$ and, consequently,~\eqref{ekinHj_fold} for $k>1$.
Similarly, $[\omega_{1:k},y]$ has to satisfy 
\begin{align*}
0_{1,j-k}=[\omega_{1:k},y]H_{:,k+1:j}&=\omega_{1:k}H_{1:k,k+1:j}+yH_{k+1:j,k+1:j}=\omega H_{1:j+1,k+1:j}\\
 &=\omega_{1:k}H_{1:k,k+1:j}+\omega_{k+1:j+1}H_{k+1:j+1,k+1:j},
\end{align*}
leading to 
\begin{align*}
 y=\omega_{k+1:j+1}H_{k+1:j+1,k+1:j}H_{k+1:j,k+1:j}^{-1}=\omega_{k+1:j} + \omega_{j+1} h_{j+1,j}e_{j-k}^*H_{k+1:j,k+1:j}^{-1},
\end{align*}
where $e_{j-k}$ is a canonical vector of length $j-k$. Hence,
\begin{align*}
 v_j^{(k)}=[\omega_{1:k},y]=\omega_{1:j}+[0_{1,k},[0_{1,j-k-1},h_{j+1,j}\omega_{j+1}]H_{k+1:j,k+1:j}^{-1}],
\end{align*}
leading to \eqref{ekinHj_fjk}. Finally, the normalization constant $\phi_{j}^{(k)}$ is obtained by the requirement $[u,y]H_{1:j,k}=1$.
\end{proof}
%%%%%%%%%%%%%%
\section{Setup of example \texttt{msd}}
\label{sec:appB}
The example \textit{msd} represents a variant of one realization of the series of examples in~\cite{TruV09}. Set $K:=\diag{I_3\otimes K_1,k_0+k_1+k_2+k_3}+k_+e_{3n}^T-e_{3n}k_+\in\R^{n_2\times n_2},~K_1:=\mathrm{tridiag}(-1,2,-1)\in\R^{n_1\times n_1}$ with $k_0=0.1$, $k_1=1$, $k_2=2$, $k_3=4$, $n_2=3n_1+1$, and $k_+:=[(k_1,k_2,k_3)\otimes e_{n_1}^T,0]^T$. Further, $M_1:=\diag{m_1I_{n_1},m_2I_{n_1},m_3I_{n_1}m_0}\in\R^{n_2\times n_2}$ with $m_0=1000$, $m_1=10$, $m_2=k_2$, $m_3=k_3$ and $D:=\alpha M+\beta(K+K(M^{-1}K)+K(M^{-1}K)^2+K(M^{-1}K)^3)+\nu[e_1,e_{n_1},e_{2n_1+1}][e_1,e_{n_1},e_{2n_1+1}]^T$, $\alpha=0.8$, $\beta=0.1$, $\nu=16$. In~\cite{TruV09} a different matrix $D$ was used involving a term $M^{\half}\sqrt{M^{-\half}KM^{-\half}}M^{\half}$ which was infeasible to set up in a large scale setting (the middle matrix square is a dense matrix). The version of $D$ used here was similarly used in~\cite{morBenKS13}. Construct the $n\times n$, $n:=2n_2$ block matrices $\hA:=\smb 0&I_{n_2}\\-K&-D\sme$, $\hM:=\diag{I_{n_2},M_1}$ representing a linearization of the quadratic matrix pencil $\lambda^2M_1+\lambda D+K$. The right hand side factor is set up as $\hB=\smb I_{2m}&0_{2m,n_2-m}&\begin{smallmatrix}0_{m,n_2-m}&I_m\\I_m&0_{m,n_2-m}\end{smallmatrix}\sme^T\in\R^{n\times 2m}$.
In Section~\ref{sec:num} we use $n_1=4000$, $m=2$ leading to $n=24002$, $r=4$. 
Finally,  $A:=P^T\hA P$, $M:=P^T\hM P$, $B:=P^T\hB$ as in \cite[Section 5.6]{KreMR17}, where $P$ is a perfect shuffle permutation: This leads to banded matrices $A,M$ except some sole off-band entries form the low-rank update of $D$, resulting in noticable computational savings when applying sparse direct solvers or computing sparse (incomplete) factorizations. An alternative to exploit the structure in $\hA,\hM$ within RKSM, LR-ADI is described in, e.g., \cite{morBenKS13, Kue16}.

\bibliographystyle{abbrv}      % mathematics and physical sciences

\begin{thebibliography}{10}

\bibitem{AhmSG12}
{\sc M.~I. Ahmad, D.~B. Szyld, and M.~B. van Gijzen}, {\em {Preconditioned
  multishift BiCG for {{\cal H}}$_2$-optimal model reduction}}, {{SIAM} J.
  Matrix Anal. Appl.}, 38 (2017), pp.~401--424.

\bibitem{Antoulas05}
{\sc A.~C. Antoulas}, {\em {Approximation of large-scale dynamical systems}},
  vol.~6 of {Advances in Design and Control}, SIAM, Philadelphia, PA, 2005.

\bibitem{BakES15}
{\sc J.~Baker, M.~Embree, and J.~Sabino}, {\em Fast singular value decay for
  {L}yapunov solutions with nonnormal coefficients}, {{SIAM} J. Matrix Anal.
  Appl.}, 36 (2015), pp.~656--668.

\bibitem{BarS72}
{\sc R.~H. Bartels and G.~W. Stewart}, {\em {Solution of the Matrix Equation
  ${AX}+{XB}={C}$: {A}lgorithm 432}}, Comm. ACM, 15 (1972), pp.~820--826.

\bibitem{BecT17}
{\sc B.~Beckermann and A.~Townsend}, {\em {On the Singular Values of Matrices
  with Displacement Structure}}, {{SIAM} J. Matrix Anal. Appl.}, 38 (2017),
  pp.~1227--1248.

\bibitem{BenBKetal18}
{\sc P.~Benner, Z.~Bujanovi{\'c}, P.~K{\"u}rschner, and J.~Saak}, {\em {{RADI}:
  A low-rank {ADI}-type algorithm for large scale algebraic {R}iccati
  equations}}, {Numer. Math.}, 138 (2018), pp.~301--330.

\bibitem{BenHSetal16}
{\sc P.~Benner, M.~Heinkenschloss, J.~Saak, and H.~K. Weichelt}, {\em {An
  inexact low-rank {N}ewton-{ADI} method for large-scale algebraic {R}iccati
  equations}}, {Appl. Numer. Math.}, 108 (2016), pp.~125--142.

\bibitem{BenK14}
{\sc P.~Benner and P.~K{\"u}rschner}, {\em {Computing Real Low-rank Solutions
  of {S}ylvester equations by the Factored {ADI} Method}}, {Comput. Math.
  Appl.}, 67 (2014), pp.~1656--1672.

\bibitem{morBenKS13}
{\sc P.~Benner, P.~K{\"u}rschner, and J.~Saak}, {\em {An Improved Numerical
  Method for Balanced Truncation for Symmetric Second Order Systems}}, {Math.
  Comput. Model. Dyn. Sys.}, 19 (2013), pp.~593--615.

\bibitem{BenKS13}
\leavevmode\vrule height 2pt depth -1.6pt width 23pt, {\em {Efficient
  {H}andling of {C}omplex {S}hift {P}arameters in the {L}ow-{R}ank {C}holesky
  {F}actor {ADI} {M}ethod}}, {Numer. Algorithms}, 62 (2013), pp.~225--251.

\bibitem{BenKS14}
\leavevmode\vrule height 2pt depth -1.6pt width 23pt, {\em {Self-Generating and
  Efficient Shift Parameters in {ADI} Methods for Large {L}yapunov and
  {S}ylvester Equations}}, {Electr. Trans. Num. Anal.}, 43 (2014),
  pp.~142--162.

\bibitem{BenLT09}
{\sc P.~Benner, R.-C. Li, and N.~Truhar}, {\em On the {ADI} method for
  {S}ylvester equations}, J. Comput. Appl. Math., 233 (2009), pp.~1035--1045.

\bibitem{BenS13}
{\sc P.~Benner and J.~Saak}, {\em {Numerical solution of large and sparse
  continuous time algebraic matrix {R}iccati and {L}yapunov equations: a state
  of the art survey}}, GAMM Mitteilungen, 36 (2013), pp.~32--52.

\bibitem{BerM17}
{\sc L.~Bergamaschi and {\'A}.~Mart{\'i}nez}, {\em {Two-stage spectral
  preconditioners for iterative eigensolvers}}, {Numer. Lin. Alg. Appl.}, 24
  (2017), p.~e2084.
\newblock e2084 nla.2084.

\bibitem{BerG15}
{\sc M.~Berljafa and S.~G{\"u}ttel}, {\em {Generalized Rational Krylov
  Decompositions with an Application to Rational Approximation}}, {{SIAM} J.
  Matrix Anal. Appl.}, 36 (2015), pp.~894--916.

\bibitem{BouF05}
{\sc A.~Bouras and V.~Frayss\'{e}}, {\em Inexact matrix-vector products in
  {K}rylov methods for solving linear systems: A relaxation strategy}, {{SIAM}
  J. Matrix Anal. Appl.}, 26 (2005), pp.~660--678.

\bibitem{CanSV14}
{\sc C.~Canuto, V.~Simoncini, and M.~Verani}, {\em {On the decay of the inverse
  of matrices that are sum of Kronecker products}}, {Linear Algebra Appl.}, 452
  (2014), pp.~21--39.

\bibitem{CroP17}
{\sc M.~Crouzeix and C.~Palencia}, {\em {The Numerical Range is a
  $(1+\sqrt{2})$-Spectral Set}}, {{SIAM} J. Matrix Anal. Appl.}, 38 (2017),
  pp.~649--655.

\bibitem{DruKS11}
{\sc V.~Druskin, L.~A. Knizhnerman, and V.~Simoncini}, {\em {Analysis of the
  rational {K}rylov subspace and {ADI} methods for solving the {L}yapunov
  equation}}, {{SIAM} J. Numer. Anal.}, 49 (2011), pp.~1875--1898.

\bibitem{DruS11}
{\sc V.~Druskin and V.~Simoncini}, {\em {Adaptive rational {K}rylov subspaces
  for large-scale dynamical systems}}, {Syst. Cont. Lett.}, 60 (2011),
  pp.~546--560.

\bibitem{FrSp08}
{\sc M.~A. Freitag and A.~Spence}, {\em {Convergence theory for inexact inverse
  iteration applied to the generalised nonsymmetric eigenproblem}}, Electron.
  Trans. Numer. Anal., 28 (2007/08), pp.~40--64.

\bibitem{FrSp09}
\leavevmode\vrule height 2pt depth -1.6pt width 23pt, {\em {Shift-invert
  {A}rnoldi's method with preconditioned iterative solves}}, SIAM J. Matrix
  Anal. Appl., 31 (2009), pp.~942--969.

\bibitem{GaaS17}
{\sc S.~Gaaf and V.~Simoncini}, {\em Approximating the leading singular
  triplets of a large matrix function}, {Appl. Numer. Math.}, 113 (2017),
  pp.~26 -- 43.

\bibitem{Gra04}
{\sc L.~Grasedyck}, {\em Existence of a low rank or $\mathcal{H}$-matrix
  approximant to the solution of a {S}ylvester equation}, Numerical Linear
  Algebra with Applications, 11 (2004), pp.~371--389.

\bibitem{Gue10}
{\sc S.~G{\"u}ttel}, {\em {Rational {K}rylov Methods for Operator Functions}},
  PhD thesis, Technische Universit{\"a}t Bergakademie Freiberg, Germany, 2010.
\newblock Available online from the Qucosa server.

\bibitem{Gue13}
\leavevmode\vrule height 2pt depth -1.6pt width 23pt, {\em {Rational {K}rylov
  approximation of matrix functions: {N}umerical methods and optimal pole
  selection}}, GAMM-Mitteilungen, 36 (2013), pp.~8--31.

\bibitem{morJaiK94}
{\sc I.~M. Jaimoukha and E.~M. Kasenally}, {\em {{K}rylov subspace methods for
  solving large {L}yapunov equations}}, {{SIAM} J. Numer. Anal.}, 31 (1994),
  pp.~227--251.

\bibitem{KreMR17}
{\sc D.~{Kressner}, S.~{Massei}, and L.~{Robol}}, {\em {Low-rank updates and a
  divide-and-conquer method for linear matrix equations}}, arXiv e-prints,
  (2017).

\bibitem{Kue16}
{\sc P.~K{\"u}rschner}, {\em {Efficient Low-Rank Solution of Large-Scale Matrix
  Equations}}, {D}issertation, {Otto-von-Guericke-Universit{\"a}t}, Magdeburg,
  Germany, Apr. 2016.

\bibitem{LehM98}
{\sc R.~B. Lehoucq and K.~Meerbergen}, {\em {Using Generalized Cayley
  Transformations within an Inexact Rational Krylov Sequence Method}}, {{SIAM}
  J. Matrix Anal. Appl.}, 20 (1998), pp.~131--148.

\bibitem{morLi00}
{\sc J.-R. Li}, {\em {Model Reduction of Large Linear Systems via Low Rank
  System {G}ramians}}, PhD thesis, Massachusettes Institute of Technology,
  September 2000.

\bibitem{LiW02}
{\sc J.-R. Li and J.~White}, {\em {Low Rank Solution of {L}yapunov Equations}},
  {{SIAM} J. Matrix Anal. Appl.}, 24 (2002), pp.~260--280.

\bibitem{ReiW12}
{\sc M.~Opmeer, T.~Reis, and W.~Wollner}, {\em {Finite-Rank ADI Iteration for
  Operator Lyapunov Equations}}, {{SIAM} J. Control Optim.}, 51 (2013),
  pp.~4084--4117.

\bibitem{PalS17}
{\sc D.~{Palitta} and V.~{Simoncini}}, {\em {Numerical methods for large-scale
  Lyapunov equations with symmetric banded data}}, {{SIAM} J. Sci. Comput.}, 40
  (2018), pp.~A3581--A3608.

\bibitem{Pen99}
{\sc T.~Penzl}, {\em {A cyclic low rank {S}mith method for large sparse
  {L}yapunov equations}}, {{SIAM} J. Sci. Comput.}, 21 (2000), pp.~1401--1418.

\bibitem{Ruh94c}
{\sc A.~Ruhe}, {\em {The {R}ational {K}rylov algorithm for nonsymmetric
  Eigenvalue problems. {III}: Complex shifts for real matrices}}, {BIT}, 34
  (1994), pp.~165--176.

\bibitem{Saa90}
{\sc Y.~Saad}, {\em {Numerical Solution of Large {L}yapunov Equation}}, in
  {Signal Processing, Scattering, Operator Theory and Numerical Methods}, M.~A.
  Kaashoek, J.~H. {van~Schuppen}, and A.~C.~M. Ran, eds., {Birkh{\"a}user},
  1990, pp.~503--511.

\bibitem{Saa09}
{\sc J.~Saak}, {\em {Efficient Numerical Solution of Large Scale Algebraic
  Matrix Equations in PDE Control and Model Order Reduction}}, PhD thesis, TU
  Chemnitz, July 2009.
\newblock Available from
  \url{http://nbn-resolving.de/urn:nbn:de:bsz:ch1-200901642}.

\bibitem{Sab07}
{\sc J.~Sabino}, {\em {Solution of Large-Scale Lyapunov Equations via the Block
  Modified Smith Method}}, PhD thesis, Rice University, Houston, Texas, June
  2007.
\newblock Available from:
  \url{http://www.caam.rice.edu/tech\_reports/2006/TR06-08.pdf}.

\bibitem{ShaSS15}
{\sc S.~D. Shank, V.~Simoncini, and D.~B. Szyld}, {\em {Efficient low-rank
  solution of generalized {L}yapunov equations}}, {Numer. Math.}, 134 (2015),
  pp.~327--342.

\bibitem{Sim05}
{\sc V.~Simoncini}, {\em {Variable accuracy of matrix-vector products in
  projection methods for eigencomputation}}, SIAM J. Numer. Anal., 43 (2005),
  pp.~1155--1174.

\bibitem{Sim07}
\leavevmode\vrule height 2pt depth -1.6pt width 23pt, {\em {A new iterative
  method for solving large-scale {Lyapunov} matrix equations}}, {{SIAM} J. Sci.
  Comput.}, 29 (2007), pp.~1268--1288.

\bibitem{Sim16}
\leavevmode\vrule height 2pt depth -1.6pt width 23pt, {\em {Analysis of the
  rational {K}rylov subspace projection method for large-scale algebraic
  {R}iccati equations}}, {{SIAM} J. Matrix Anal. Appl.}, 37 (2016),
  pp.~1655--1674.

\bibitem{Simoncini16}
\leavevmode\vrule height 2pt depth -1.6pt width 23pt, {\em {Computational
  methods for linear matrix equations}}, SIAM Rev., 58 (2016), pp.~377--441.

\bibitem{SimSz03}
{\sc V.~Simoncini and D.~B. Szyld}, {\em {Theory of inexact {K}rylov subspace
  methods and applications to scientific computing}}, {{SIAM} J. Sci. Comput.},
  25 (2003), pp.~454--477.

\bibitem{SimSM14}
{\sc V.~Simoncini, D.~B. Szyld, and M.~Monsalve}, {\em {On two numerical
  methods for the solution of large-scale algebraic Riccati equations}}, {{IMA}
  J. Numer. Anal.}, 34 (2014), pp.~904--920.

\bibitem{Sood15}
{\sc K.~Soodhalter}, {\em {A block MINRES algorithm based on the banded Lanczos
  method}}, Numer. Algorithms, 69 (2015), pp.~473--494.

\bibitem{Sun08}
{\sc K.~Sun}, {\em {Model order reduction and domain decomposition for
  large-scale dynamical systems}}, PhD thesis, Rice University, Houston, 2008.
\newblock Available from \url{http://search.proquest.com/docview/304507831}.

\bibitem{TruV07}
{\sc N.~Truhar and K.~Veseli{\'c}}, {\em Bounds on the trace of a solution to
  the {L}yapunov equation with a general stable matrix}, {Syst. Cont. Lett.},
  56 (2007), pp.~493--503.

\bibitem{TruV09}
\leavevmode\vrule height 2pt depth -1.6pt width 23pt, {\em An efficient method
  for estimating the optimal dampers' viscosity for linear vibrating systems
  using {L}yapunov equation}, {{SIAM} J. Matrix Anal. Appl.}, 31 (2009),
  pp.~18--39.

\bibitem{EstS04}
{\sc J.~van~den Eshof and G.~Sleijpen}, {\em {Inexact Krylov Subspace Methods
  for Linear Systems}}, {{SIAM} J. Matrix Anal. Appl.}, 26 (2004),
  pp.~125--153.

\bibitem{Wac13}
{\sc E.~L. Wachspress}, {\em {The {ADI} Model Problem}}, Springer New York,
  2013.

\bibitem{WolP16}
{\sc T.~Wolf and H.~K.-F. Panzer}, {\em {The {ADI} iteration for {L}yapunov
  equations implicitly performs $\mathcal{H}_2$ pseudo-optimal model order
  reduction}}, Internat. J. Control, 89 (2016), pp.~481--493.


\end{thebibliography}

\end{document}